\documentclass[11pt,reqno]{article}
\usepackage[left=2.5cm,top=3cm,right=2.5cm,bottom=3cm,bindingoffset=0.5cm]{geometry}

\usepackage{dsfont, amssymb,amsmath,amscd,latexsym, amsthm, amsxtra,amsfonts}
\usepackage{lineno}
\usepackage[all]{xy}
\usepackage[active]{srcltx}

\usepackage{tikz}
\usepackage[round]{natbib}
\usepackage{bbm}
\usepackage{enumerate}
\usepackage{mathrsfs}
\usepackage{graphicx}
\usepackage{comment}
\usepackage{mathtools}
\usepackage{cases}
 \usepackage{tikz}
\usetikzlibrary{calc,arrows}
\usepackage{verbatim}
\usepackage{graphicx}
\usepackage{subfigure}
\usepackage{color}

\usepackage{tikz}
\usetikzlibrary{calc,arrows}
\usepackage{verbatim}
\usepackage{graphicx}
\usepackage{subfigure}
\usepackage{epstopdf}
\usepackage{verbatim}
\usepackage{graphicx}
\usepackage{subfigure}
\usepackage{color}
\usepackage{algorithm,algorithmic}

\newtheorem{theorem}{Theorem}[section]

\newtheorem{proposition}[theorem]{Proposition}
\newtheorem{lemma}[theorem]{Lemma}

\theoremstyle{definition}
\newtheorem{definition}[theorem]{Definition}

\theoremstyle{definition}

\theoremstyle{definition}
\newtheorem{remark}{Remark}
\theoremstyle{definition}
\newtheorem{assumption}{Assumption}

\renewcommand{\epsilon}{\varepsilon}

\usepackage[pdfstartview=FitH, bookmarksnumbered=true,bookmarksopen=true, colorlinks=true, citecolor=blue, linkcolor=blue,urlcolor=blue]{hyperref}
\usepackage{graphics}
\allowdisplaybreaks[4]

\graphicspath{{figures/}}
\title{A Reinforcement Learning Framework for Some\\ Singular Stochastic Control Problems}

\author{Zongxia Liang\thanks{Department of Mathematical Sciences, Tsinghua University, Beijing 100084, People’s Republic of China
  (\url{liangzongxia@tsinghua.edu.cn}).}
  \and  Xiaodong Luo\thanks{Department of Mathematical Sciences, Tsinghua University, Beijing 100084, People’s Republic of China
  (\url{luoxd21@mails.tsinghua.edu.cn}).}
  \and Xiang Yu\thanks{Department of Applied Mathematics, The Hong Kong Polytechnic University, Kowloon, Hong Kong (\url{xiang.yu@polyu.edu.hk})}
}
\date{}  
\numberwithin{equation}{section}

\usepackage{amsopn}

\begin{document}

\maketitle

\vspace{-0.2in}
\begin{abstract}
We develop a continuous-time reinforcement learning framework for a class of singular stochastic control problems without entropy regularization. The optimal singular control is characterized as the optimal singular control law, which is a pair of regions of time and the augmented states. The goal of learning is to identify such an optimal region via the trial-and-error procedure. In this context, we generalize the existing policy evaluation theories with regular controls to learn our optimal singular control law and develop a policy improvement theorem via the region iteration. To facilitate the model-free policy iteration procedure, we further introduce the zero-order and first-order q-functions arising from singular control problems and establish the martingale characterization for the pair of q-functions together with the value function. Based on our theoretical findings, some q-learning algorithms are devised accordingly and a numerical example based on simulation experiment is presented.

\vskip 5pt\noindent
\end{abstract}

\section{Introduction}
The singular stochastic control problem is an important class of control problems that features the displacement
of the states to be possibly discontinuous, which dates back to the pioneer study of the fuel follower problem of spacecraft; see, e.g., \cite{bather1967sequential}, \cite{benevs1980some}. The main difference between singular control and regular control is that the monotonicity of control is additionally required for singular control, which makes it suitable to model problems that exhibit a monotone or irreversible feature, such as the optimal dividend problem in  \cite{asmussen1997controlled}, \cite{jeanblanc1995optimization}, the irreversible investment or reinsurance in \cite{kogan2001equilibrium}, \cite{yan2022irreversible}, and the extraction of natural resources in \cite{ferrari2021optimal}, among others.
\vskip 5pt
The conventional methods to study the continuous-time stochastic regular control and singular control problems assume the full knowledge of the model. In reality, however, the decision maker may only have a limited or no information about the model or model parameters. Reinforcement learning (RL) provides a powerful way for the agent to learn the optimal control through the trial-and-error procedure. By taking an action and observing the resulting state signal and the reward outcome, the agent can learn to select new actions gradually, eventually approximating the optimal control. Recently, the theoretical foundations for continuous-time RL algorithms in stochastic regular control problems have been developed in a series of studies in  \cite{jia2022policy}, \cite{jia2022policy1}, \cite{jia2023q}, \cite{wang2020reinforcement},\cite{wang2020continuous} in an entropy regularized formulation. In particular, the Shannon entropy is considered therein to encourage the exploration to ensure the optimal policy to be a randomized policy. Later, this exploitation and exploration formulation with entropy regularization has been quickly employed and generalized to address various financial applications with regular controls, to name a few, see \cite{DaiDJZ} for RL method to solve Merton’s utility maximization problem in an incomplete market; \cite{BoHuangYu} for continuous-time q-learning method to address the optimal tracking portfolio problems with state reflections; \cite{WeiYu} and \cite{WYY25} for the continuous-
time q-learning algorithms in mean-field control and mean-field game problems; \cite{Jia24} for the continuous-time risk-sensitive reinforcement learning theories and algorithms; and \cite{HLYZ25} on continuous-time reinforcement learning for optimal switching over multiple regimes.
\vskip 5pt
Despite the fast-growing literature on continuous-time RL for regular control problems with randomized policy, the development of RL methods for singular control problems is underexplored due to the lack of understanding on the proper entropy regularization of the monotone singular controls. Only some related recent studies in the context of optimal stopping can be found, which essentially transform the optimal stopping problems to regular control problems such that the entropy regularization can be exercised to randomize the policy. \cite{dong2024randomized} formulates the randomized stopping problem as a jump intensity control problem and devise the RL algorithm within a regular control framework. \cite{dianetti2024exploratory} formulates the randomization of the stopping time as a singular control problem with residual entropy and introduce the augmented state process as the probability of stopping before time $t$ to facilitate the policy iteration. \cite{dai2024learning} proposes the penalty approximation of the stopping problem first as a regular control problem such that the discrete type entropy regularization can be incorporated therein to encourage the randomization.  

\vskip 5pt

Unlike the optimal stopping problems, the singular stochastic control problem is generically more difficult to formulate a tractable entropy regularization to randomize the singular control process. Even the general randomization is possible, the implicit characterization may not facilitate the task of policy iteration that plays the key role in RL algorithms. In the present paper, we  consider a RL method without entropy regularization, but contribute to the policy iteration via the region iteration arising from the singular control problem and HJB characterization.  We focus on the one-dimensional setting where the singular control is learned by identifying the singular control law, which is the region of time $t$ and the augmented state processes that generates the admissible singular control. First, we extend the policy evaluation of regular control based on q-function and martingale characterization in \cite{jia2023q} to our singular control setting based on the regions of singular control law. Second, we establish a policy improvement theorem based on the region iterations of singular control laws. It is shown that the optimal singular control law can be related to the zero-order and first-order q-functions for singular control problems, where the first-order q-function coincides with the q-function for regular control problems; see \cite{jia2023q}. We establish the martingale characterization of this pair of q-functions and the value function in the context of learning the regions of singular control law. Finally, we devise the continuous-time RL algorithms in both finite-time and infinite-time settings and illustrate the effectiveness of the algorithm to learn the optimal singular control in a concrete example. The singular control policy in the present paper corresponds to the region of singular control law, which is not randomized in our proposed  region iteration method. To encourage the exploration, our algorithm can be combined with other exploration methods such as $\epsilon$-greedy algorithms to help the convergence of iterations towards to the optimal solution. On the other hand, how to design a proper entropy regularization framework to directly randomize the singular control with a tractable characterization or iteration rule is still an open problem, which will be left for the future investigation. 
\vskip 5pt
The rest of the paper is organized as follows. In Section \ref{preli}, we introduce the mathematical setup and preliminary results of the finite time singular control problem, along with the region characterization of the singular control, i.e., the singular control law. Section \ref{theory} establishes the main results for continuous-time RL of the singular control problem, including the PE theorem, PI theorem, and martingale characterization of the pair of q-functions. In Section \ref{algorithm}, we apply the aforementioned q-learning theories to devise some RL algorithms for singular control problems, and one simple numerical example with simulation experiment is presented therein.

\section{Problem Formulation and Singular Control Law}
\label{preli}
The goal of this paper is to develop a reinforcement learning method for a class of singular control problems in a one-dimensional setting originating from the irreversible reinsurance; see, e.g.,  \cite{liang2024equilibria}, \cite{yan2022irreversible}. In this section, we first introduce the mathematical formulation of the singular control problem and recall some preliminary results when the model is fully known, including the verification theorem and a uniqueness result.  

\subsection{Problem formulation and standing assumptions}
We first introduce a class of singular control problems arising from irreversible reinsurance, following the terms and assumptions in \cite{liang2024equilibria}. 

To start with, let $(\Omega,\mathcal{F},\{\mathcal{F}_{t}\}_{t\ge 0}, \mathbb{P})$ be a filtered probability space satisfying the usual conditions, on which there exists a standard one-dimensional Brownian motion $\{B_{t}\}_{t\ge 0}$. The state process under the singular control $\xi$ is denoted by $X^{\xi}=\{X^{\xi}_{t}\}_{t\ge 0}$, which is governed by 
\begin{equation}
\label{dyna}
\begin{cases}
dX^{\xi}_{t}=\mu(X^{\xi}_{t},t,\xi_{t})dt+\sigma(X^{\xi}_{t},t,\xi_{t})dB_{t}-d\xi_{t},\quad \forall t\ge 0,\\
X^{\xi}_{0-}=x_{0},
\end{cases}
\end{equation}
where the singular control $\xi=\{\xi_{t}\}_{t\ge 0}$ is required to be non-decreasing and c\`{a}dl\`{a}g with initial value $\xi_{0-}=0$. The drift $\mu(x,t,y)$ and the volatility $\sigma(x,t,y)$ are deterministic functions, and $\sigma(x,t,y)$ is assumed to be non-negative. In the context of irreversible reinsurance, $X^{\xi}$ and $\xi$ represent the risk exposure and the accumulated reinsurance coverage of risk exposure, respectively.

\begin{remark}
The current method only works for a one-dimensional state process and a one-dimensional singular control process. The extension to the multi-dimensional case will be left for future research. 
\end{remark}

In the finite-horizon setting, the optimal singular control problem is given by
\begin{equation}
\begin{aligned}
\label{object}
&\min_{\xi }J(x,t,y;\xi):=\mathbb{E}_{x,t,y}\bigg[\int_{t}^{T}e^{-\beta(r-t)}H(X^{\xi}_{r},r,\xi_{r})dr+e^{-\beta(T-t)}F(X^{\xi}_{T},\xi_{T})\\&+\int_{t}^{T}e^{-\beta(r-t)}c(X^{\xi}_{r},r,\xi_{r})\circ d\xi_{r}\bigg],\quad\forall
(x,t,y)\in\mathcal{Q}:=\mathbb{R}\times[0,T]\times[0,+\infty),
\end{aligned}
\end{equation}
where $T$ is the terminal time horizon and $\beta\geq 0$ denotes the discount rate. To align with the convention in the RL literature, we shall call $J$ the value function under the singular control $\xi$. 

In Problem (\ref{object}),  $\mathbb{E}_{x,t,y}$ denotes the expectation conditioning on  $X_{t-}=x$, $\xi_{t-}=y$. We also note that 
\begin{equation*}
\int_{t}^{T}\!e^{-\beta(r-t)}c(X^{\xi}_{r},r,\xi_{r})\circ d\xi_{r}\!:=\!\int_{t}^{T}\!e^{-\beta(r-t)}c(X^{\xi}_{r},r,\xi_{r})d\xi^{c}_{r}\!+\!\sum\limits_{r\in[t,T]}\!\int_{0}^{\Delta\xi_{r}}\!e^{-\beta(r-t)}c(X^{\xi}_{r-}\!-\!u,r,\xi_{r-}\!+\!u)du 
\end{equation*}
with $\xi^{c}_{r}$ and $\Delta\xi_{r}:=\xi_{r}-\xi_{r-}$ being the continuous and jump parts of $\xi_{r}$ at time $r$. This type of definition is exactly the form that can be continuously approximated by regular controls, which leads to more compact and convenient mathematical representations in our later analysis. This definition has been commonly adopted in the literature, e.g., \cite{zhu1992generalized}, \cite{lon2011model}, \cite{becherer2018optimala}, \cite{ferrari2025singular}, among others.

We also let non-negative functions $H(x,t,y)$, $F(x,y)$ and $c(x,t,y)$ represent the running cost, the terminal cost and control cost, respectively, with $F$ satisfying that the function $a\mapsto F(x-a,y+a)+\int_{0}^{a}c(x-u,T,y+u)du$ is convex with finite minimum point. As a result, the optimal singular control exists for the problem at $(x,T,y)$, specifically, the optimal singular control at the terminal time $T $ is given by
\begin{equation*}
\hat{\xi}_{T}=y+\mathop{\arg\min}\limits_{a\ge 0}\Big\{F(x-a,y+a)+\int_{0}^{a}c(x-u,T,y+u)du\Big\},
\end{equation*}
and the optimal value function is
\begin{equation*}
J(x,T,y;\hat{\xi})=\tilde{F}(x,y):=\min\limits_{a\ge 0}\Big\{F(x-a,y+a)+\int_{0}^{a}c(x-u,T,y+u)du\Big\}.
\end{equation*}

In addition, we also impose the following standing assumption.

\begin{assumption}
\label{as1}
$\mu,\sigma,H,F,c$ are continuous functions. Moreover, $\mu$ and  $\sigma$ have linear growth in $x+y$, and $H,F$ and $c$ have polynomial growth in $x+y$, i.e., for $\varphi_{1}=\mu$,  $\sigma$, and $\varphi_{2}=H$, $F$,  $c$ (where $F(x,t,y):=F(x,y)$), there exist a constant $C>0$ and a positive integer $k\ge 1$ such that
\begin{equation*}
\begin{aligned}
&|\varphi_{1}(x,t,y)|\le C(1+|x+y|),\quad \forall (x,t,y)\in\mathcal{Q},\\
&|\varphi_{2}(x,t,y)|\le C(1+|x+y|^{k}),\quad \forall (x,t,y)\in\mathcal{Q}.
\end{aligned}
\end{equation*}
\end{assumption}

We first have the following estimation result, whose proof is similar to that of  Lemma 2 in \cite{jia2022policy1}.

\begin{lemma}
\label{l1}
Under Assumption \ref{as1}, if an $\{\mathcal{F}_{r}\}_{r\in[t, T]}$-adapted c\`{a}dl\`{a}g process $X=\{X_{r}\}_{r\in[t, T]}$ and an $\{\mathcal{F}_{r}\}_{r\in[t, T]}$-adapted non-decreasing c\`{a}dl\`{a}g process $\xi=\{\xi_{r}\}_{r\in[t, T]}$ satisfy
\begin{equation*}
\begin{cases}
dX_{r}=\mu(X_{r},r,\xi_{r})dr+\sigma(X_{r},r,\xi_{r})dB_{r}-d\xi_{r},\quad r\in[t,T],\\
X_{t-}=x,\quad \xi_{t-}=y,
\end{cases}
\end{equation*}
then for any $k\ge 2$, there exists a $C=C(k)$ such that
\begin{equation*}
\mathbb{E}_{x,t,y}\left \{\max\limits_{r\in[t,T]}|X_{r}+\xi_{r}|^{k}\right\}\le C(1+|x+y|^{k}).
\end{equation*}
\end{lemma}

We next give the definition of admissible singular controls.
\begin{definition}
\label{ads}
An $\{\mathcal{F}_{r}\}_{r\in[t, T]}$-adapted non-decreasing c\`{a}dl\`{a}g process $\xi=\{\xi_{r}\}_{r\in[t, T]}$ is called an admissible singular control if the following properties are satisfied:\\
(a) The stochastic differential equation
	\begin{equation}
    \label{dyna1}
	\begin{cases}
dX^{\xi}_{r}=\mu(X^{\xi}_{r},r,\xi_{r})dr+\sigma(X^{\xi}_{r},r,\xi_{r})dB_{r}-d\xi_{r},\quad \forall r\in[t,T],\\
	X^{\xi}_{t-}=x
	\end{cases}
	\end{equation}
	has a unique strong solution $\{X^{\xi}_{r}\}_{r\in[t,T]}$.\\
	(b) $\{X^{\xi}_{r}\}_{r\in[t,T]}$ given in (a) satisfies
    \begin{equation*}
	\begin{aligned}
	&\mathbb{E}_{x,t,y}\int_{t}^{T}e^{-\beta(r-t)}H(X^{\xi}_{r},r,\xi_{r})dr<\infty,\\
    &\mathbb{E}_{x,t,y}F(X^{\xi}_{T},\xi_{T})<\infty,\\ &\mathbb{E}_{x,t,y}\int_{t}^{T}e^{-\beta(r-t)}c(X^{\xi}_{r},r,\xi_{r})\circ d\xi_{r}<\infty,
	\end{aligned}
    \end{equation*}
where $y=\xi_{t-}$ is the initial value of $\xi$.\\
Moreover, we denote the set of all admissible singular controls on $[t, T]$ by $\mathcal{D}_{[t, T]}$.
\end{definition}

If $\mu$ and $\sigma$ are additionally assumed to be Lipschitz w.r.t the first variable, the following result shows that $\xi$ is admissible under some integrability conditions.

\begin{proposition}
\label{propads}
Suppose that $\mu$ and $\sigma$ are Lipschitz in the first variable, then for any $\{\mathcal{F}_{r}\}_{r\in[t,T]}$-adapted non-decreasing c\`{a}dl\`{a}g process $\xi=\{\xi_{r}\}_{r\in[t,T]}$, the SDE (\ref{dyna1}) has a unique strong solution $X^{\xi}$. Moreover, if $\xi$ satisfies that $\mathbb{E}[(\xi_{T}-\xi_{t-})^{\alpha}|\mathcal{F}_{t-}]<\infty$ for some $\alpha>1$, then $\xi$ is admissible.
\end{proposition}
\begin{proof}
Fix $(x,t,y)\in\mathcal{Q}$. Using the transform $Y^{\xi}_{r}:=X^{\xi}_{r}+\xi_{r}$, we rewrite the SDE for $Y^{\xi}$ as
\begin{equation*}
\begin{cases}
dY^{\xi}_{r}(w)=\bar{\mu}(Y^{\xi}_{r}(w),r,w)dr+\bar{\sigma}(Y^{\xi}_{r}(w),r,w)dB_{r}(w),\quad \forall r\in[t,T],\\
Y^{\xi}_{t-}(w)=x+y,
\end{cases}
\end{equation*}
where
\begin{equation*}
\begin{aligned}
\bar{\mu}(y,r,w):=\mu(y-\xi_{r}(w),r,\xi_{r}(w)),\\
\bar{\sigma}(y,r,w):=\sigma(y-\xi_{r}(w),r,\xi_{r}(w)).
\end{aligned}
\end{equation*}
The Lipschitz property of $\mu$ and $\sigma$ yields that
\begin{equation*}
\begin{aligned}
&|\bar{\mu}(y_{1},r,w)-\bar{\mu}(y_{2},r,w)|\le L|y_{1}-y_{2}|,\\
&|\bar{\sigma}(y_{1},r,w)-\bar{\sigma}(y_{2},r,w)|\le L|y_{1}-y_{2}|.
\end{aligned}
\end{equation*}
Furthermore, the linear growth assumption of $\mu$ and $\sigma$ gives that
\begin{equation*}
\begin{aligned}
|\bar{\mu}(y,r,w)|\le C(1+|y|),\\
|\bar{\sigma}(y,r,w)|\le C(1+|y|).
\end{aligned}
\end{equation*}
\textit{(i) Uniqueness:}
Let $Y^{(1)}$ and $Y^{(2)}$ be two strong solutions to the SDE given $\xi$. Then,
\begin{equation*}
\begin{aligned}
&\mathbb{E}\Big[\big[Y^{(1)}_{r}(w)-Y^{(2)}_{r}(w)\big]^{2}\Big|\mathcal{F}_{t-}\Big]\\
=&\mathbb{E}\Big[\big[\int_{t}^{r}(\bar{\mu}(Y^{(1)}_{r'}(w),r',w)-\bar{\mu}(Y^{(2)}_{r'}(w),r',w))dr'+\int_{t}^{r}(\bar{\sigma}(Y^{(1)}_{r'}(w),r',w)-\bar{\sigma}(Y^{(2)}_{r'}(w),r',w))dB_{r'}\big]^{2}\Big|\mathcal{F}_{t-}\Big]\\
\le &2(r-t+1)L^{2}\mathbb{E}\Big[\int_{t}^{r}(Y^{(1)}_{r'}(w)-Y^{(2)}_{r'}(w))^{2}dr'\Big|\mathcal{F}_{t-}\Big],\quad\forall r\in[t,T].
\end{aligned}
\end{equation*}
Then the non-negative function $v(r):=\mathbb{E}\Big[\big[Y^{(1)}_{r}(w)-Y^{(2)}_{r}(w)\big]^{2}\Big|\mathcal{F}_{t-}\Big]$ satisfies
\begin{equation*}
v(r)\le 2(r-t+1)L^{2}\int_{t}^{r}v(r')dr',\quad \forall r\in[t,T].
\end{equation*}
By the Gronwall inequality, we have $v(r)=0$ for any $r\in[t,T]$ and the uniqueness follows.\\

\noindent
\textit{(ii) Existence:} Define $Y^{(0)}_{r}(w)=x+y,\forall r\in[t,T]$ and $Y^{(k)}_{r}(w),k\ge 1$ inductively by
\begin{equation*}
Y^{(k)}_{r}(w):=x+y+\int_{t}^{r}\bar{\mu}(Y^{(k-1)}_{r'}(w),r',w)dr'+\int_{t}^{r}\bar{\sigma}(Y^{(k-1)}_{r'}(w),r',w)dB_{r'},\quad\forall k\ge 1
\end{equation*}
By the adaptiveness of $\xi$, all the $\{Y^{(k)}_{r}(w)\}_{r\in[t,T]}$ for $k\ge 0$ are $\{\mathcal{F}_{r}\}_{r\in[t,T]}$-adapted. 

Similar to the proof of uniqueness, we obtain that $ \forall\  k\ge 1,\  r\in[t,T]$,
\begin{equation}
\label{induct}
\mathbb{E}\Big[\big[Y^{(k+1)}_{r}(w)-Y^{(k)}_{r}(w)\big]^{2}\Big|\mathcal{F}_{t-}\Big]\le 2(r-t+1)L^{2}\int_{t}^{r}\mathbb{E}\Big[\big[Y^{(k)}_{r}(w)-Y^{(k-1)}_{r}(w)\big]^{2}\Big|\mathcal{F}_{t-}\Big].
\end{equation}
By the linear growth property, we have
\begin{equation*}
\begin{aligned}
&\mathbb{E}\Big[\big[Y^{(1)}_{r}(w)-Y^{(0)}_{r}(w)\big]^{2}\Big|\mathcal{F}_{t-}\Big]\\
\le&2\mathbb{E}\Big[\big[\int_{t}^{r}\bar{\mu}(Y^{(0)}_{r'}(w),r',w)dr'\big]^{2}\Big|\mathcal{F}_{t-}\Big]+2\mathbb{E}\Big[\big[\int_{t}^{r}\bar{\sigma}(Y^{(0)}_{r'}(w),r',w)dB_{r'}\big]^{2}\Big|\mathcal{F}_{t-}\Big]\\
\le& 2(r-t+1)\int_{t}^{r}C^{2}(1+x+y)^{2}dr'\\
\le& \tilde{C}(r-t),\quad \forall r\in[t,T],
\end{aligned}
\end{equation*}
where $\tilde{C}$ is a constant depending on $t,T,C,x,y$. Then, it follows from (\ref{induct}) that
\begin{equation*}
\mathbb{E}\Big[\big[Y^{(k+1)}_{r}(w)-Y^{(k)}_{r}(w)\big]^{2}\Big|\mathcal{F}_{t-}\Big]\le\tilde{C}'^{k+1}\frac{(r-t)^{k+1}}{(k+1)!},
\end{equation*}
where $\tilde{C}'$ is a constant depending on $t,T,C,x,y,L$. Then one can apply the standard argument in \cite{oksendal1992stochastic} for the regular control case to show that $Y^{(k)}_{r}$ is uniformly convergent a.s. on $[t, T]$ and the limit constitutes a strong solution of the SDE.\\

\noindent 
\textit{(iii) Admissibility:} With the moment estimates in Lemma \ref{l1}, the first two inequalities in (b) hold. The last inequality in (b) also holds because
\begin{equation*}
\begin{aligned}
\mathbb{E}_{x,t,y}&\int_{t}^{T}e^{-\beta(r-t)}c(X^{\xi}_{r},r,\xi_{r})\circ d\xi_{r}
\le \mathbb{E}_{x,t,y}\bigg\{(\xi_{T}-y)\max\limits_{r\in[t,T],u\in[0,\Delta\xi_{r}]}c(X^{\xi}_{r-}-u,r,\xi_{r-}+u)\bigg\}\\
\le&\Big[\mathbb{E}_{x,t,y}(\xi_{T}-y)^{\alpha}\Big]^{\frac{1}{\alpha}}\Big[\mathbb{E}_{x,t,y}\max\limits_{r\in[t,T],u\in[0,\Delta\xi_{r}]}c(X^{\xi}_{r-}-u,r,\xi_{r-}+u)^{\frac{\alpha}{\alpha-1}}\Big]^{\frac{\alpha-1}{\alpha}}
<\infty,
\end{aligned}
\end{equation*}
which completes the proof.
\end{proof}

The optimal singular control and the corresponding optimal value function are defined as follows.
\begin{definition}
\label{ops}
Given $(x,t,y)\in\mathcal{Q}$, suppose that $\hat{\xi}$ is an admissible singular control satisfying that, for any admissible $\xi$,
\begin{equation*}
J(x,t,y;\hat{\xi})\le J(x,t,y;\xi),
\end{equation*}
then we call $\hat{\xi}$ an optimal singular control for the problem starting at $(x,t,y)$, and  call $J(x,t,y;\hat{\xi})$ the optimal value function.
\end{definition}

It is important to note that the optimal value function is unique but there might be many optimal singular controls. Next, we introduce the notion of singular control law proposed in \cite{liang2024equilibria}, which will play a core role in our RL algorithms.

\begin{definition}[admissible singular control law]
\label{adl}
Denote $\bar{A}$ and $A^{c}$ the closure and the complement of an arbitrary set $A$. A singular control law $\Xi=(W^{\Xi},(W^{\Xi})^{c})$, with $W^{\Xi}$ being open, is a partition of the space $\mathcal{Q}$. The singular control law $\Xi$ is called admissible if the following two conditions (a) and (b)  hold:\\
 (a) Given any initial $(x,t,y)\in\mathcal{Q}$, the Skorokhod reflection problem
	\begin{equation}
	\left\{
	\begin{array}{l}
dX^{\xi}_{r}=\mu(X^{\xi}_{r},r,\xi_{r})dr+\sigma(X^{\xi}_{r},r,\xi_{r})dB_{r}-d\xi_{r},\quad \forall r\in[t,T],\\
	(X^{\xi}_{r},r,\xi_{r})\in\overline{W^{\Xi}},\quad \forall r\in[t,T],\\
	\xi_{r}=y+\int_{t}^{r}1_{\{(X^{\xi}_{u},u,\xi_{u})\in (W^{\Xi})^{c}\}}d\xi_{u},\quad \forall r\in[t,T],\\ 
	X^{\xi}_{t-}=x,\ \xi_{t-}=y
	\end{array}
	\right.
	\label{dynamic1}
	\end{equation}
	has a unique strong solution $(X^{\xi},\xi):=(X^{x,t,y,\Xi},\xi^{x,t,y,\Xi})$. We call $\xi$ the \emph{singular control generated by $\Xi$ at $(x,t,y)$}.\\
	(b) The generated singular control $\{\xi_{r}\}_{r\in[t,T]}$ is admissible in Definition \ref{ads}.

\end{definition}

The above definition allows possibly multiple sub-regions of $W^{\Xi}$ that are not necessarily connected. That is, the boundary curve between $W^{\Xi}$ and $(W^{\Xi})^{c}$ is not restricted to be a single curve. 

\begin{remark}
The terminology of ``control law'' has been originally used to denote the feedback controls in the closed-loop formulation for time-inconsistent control problems; see \cite{bjork2017timeinconsistent} or \cite{bjork2021time}. Similar terminology of ``singular control law'' was introduced in \cite{liang2024equilibria}, emphasizing the parallel role of such reflected policy for time-inconsistent singular control problems. For regular controls, the ``control law'' is a function on the state-time space; For singular controls, the ``singular control law'' is a characterization of region in the state-time-control space; both control laws can generate controls through the SDE or the Skorokhod problem. Here, we adopt such terminology for reinforcement learning.
\end{remark}

\begin{remark}
When $\mu$ and $\sigma$ are Lipschitz in the first and last variables, the Skorokhod reflection problem (\ref{dynamic1}) has a unique strong solution under some standard assumptions on the reflection region $W^{\Xi}$. Denote the unit vector of the reflection direction by $\overrightarrow{\kappa}:=(-\frac{1}{\sqrt{2}},0,\frac{1}{\sqrt{2}})$. The main assumptions can be briefly summarized as follows in our case:
\begin{itemize}
\item[(1)] The boundary $\partial W^{\Xi}$ is sufficiently smooth, such that locally it can be transformed into a half-space or by a Lipschitz mapping.
\item[(2)] There exists $\pi:\mathcal{Q}\rightarrow \overline{W^{\Xi}}$ such that if $z:=(x,t,y)\in W^{\Xi}$, then $\pi(z)=z$, and if $z\in (W^{\Xi})^{c}$, then $\pi(z)\in\partial W^{\Xi}$ and $z-\pi(z)=\alpha (-1,0,1)$ for some $\alpha\le 0$.
\item[(3)] The unit inward normal $\overrightarrow{n}(z), \forall z\in \partial W^{\Xi}$ is smooth in $z$, and there exists a positive constant $\epsilon>0$ such that for any $z\in \partial W^{\Xi}$, it holds that  $\Big(\overrightarrow{n}(z),\overrightarrow{\kappa}\Big)\ge \epsilon$.
\end{itemize}
The proof is similar to the one-dimensional reflection case where $\mu$ and $\sigma$ only depend on the state variable $x$, which mainly consists of two steps:\\
\textit{Step 1:} Apply the localization argument in \cite{anderson1976small} and \cite{dupuis1991lipschitz} to show that,  given an arbitrary c\`{a}dl\`{a}g $\mathbb{R}^{3}$ stochastic process $Z$ in $[t,T]$ with initial $Z_{t-}=(x,t,y)$, the Skorokhod reflection problem
\begin{equation*}
\begin{cases}
Z_{r}=Z'_{r}+K_{r}, \quad\forall r\in[t,T],\\
Z'_{r}\in \overline{W^{\Xi}},\quad \forall r\in[t,T],\\
K_{r}=\int_{t}^{r}\overrightarrow{\kappa}d\xi'_{r'},\quad \forall r\in[t,T],\\
\xi'_{r}=\int_{t}^{r}1_{\{Z'_{r'}\notin W^{\Xi}\}}d\xi'_{r'},\quad \forall r\in[t,T]
\end{cases}
\end{equation*}
has a unique solution $(Z',K)$, and the solution mapping $Z\rightarrow Z'$ is Lipschitz. \\
\textit{Step 2:} For an arbitrary c\`{a}dl\`{a}g $\mathbb{R}^{1}$ stochastic process $\tilde{X}$ with $\tilde{X}_{t-}=x$, take $Z_{r}:=(\tilde{X}_{r},r,y),\forall r\in[t,T]$ in the last step, denote $Z'_{r}=(X_{r},r,\xi_{r})$, then the last step implies that $\tilde{X}_{r}=X_{r}+\xi_{r}-y,\forall r\in[t,T]$, and the mapping $\mathcal{I}:\tilde{X}\rightarrow X$, $\mathcal{J}:\tilde{X}\rightarrow\xi$ are Lipschitz. As a result, the SDE
\begin{equation*}
\begin{cases}
d\tilde{X}_{r}=\mu(\mathcal{I}(\tilde{X})_{r},r,\mathcal{J}(\tilde{X})_{r})dr+\sigma(\mathcal{I}(\tilde{X})_{r},r,\mathcal{J}(\tilde{X})_{r})dB_{r},\quad \forall r\in[t,T],\\
\tilde{X}_{t-}=x
\end{cases}
\end{equation*}
has a unique strong solution $\tilde{X}$. Then $(X,\xi):=(\mathcal{I}(\tilde{X}),\mathcal{J}(\tilde{X}))$ is the unique strong solution to (\ref{dynamic1}).\\
We emphasize that the proof of \textit{Step 1} relies on a localization argument in Section 1.3 of \cite{anderson1976small} and Section 5.4 of \cite{dupuis1991lipschitz}. We refer the detailed smoothness assumptions and proof of \textit{Step 1} to \cite{dupuis1991lipschitz} and \cite{anderson1976small}.

\end{remark}

\begin{remark}
A sufficient condition for $\Xi$ to be admissible is that the solution $(X,\xi)$ in the last remark satisfies $\mathbb{E}[(\xi_{T}-\xi_{t-})^{\alpha}|\mathcal{F}_{t-}]<\infty$ for some $\alpha>1$. This holds when $\mu$, $\sigma$ are bounded, and all local regions where $\partial W^{\Xi}$ can be transformed into a half-space share the same radius. The above sufficient condition ensures that only finite local regions are reached during $[t, T]$, and the $\alpha$-moment for any $\alpha>1$ is finite in each local region using the explicit expression of the $\xi$ process in half-space reflection.
\end{remark}

With a little abuse of notation, we can define the value function of a given singular control law by $J(x,t,y;\Xi)=J(x,t,y;\xi^{x,t,y,\Xi})$ and the optimal singular control law is the one that minimizes $J(x,t,y;\Xi)$ over  the admissible singular control laws.

\subsection{Preliminary results}
In this subsection, we introduce two verification theorems: one for the optimal singular control and one for an arbitrary admissible singular control. These theorems will be useful later in reinforcement learning.

To better state the assumptions for the main results of the paper, we first introduce some regularity conditions.
\begin{definition}
The function set $\mathcal{R}^{\Xi}$ w.r.t a given admissible singular control law $\Xi$, is defined by
\begin{equation*}
\begin{aligned}
\mathcal{R}^{\Xi}:=\{\varphi:\mathcal{Q}\rightarrow\mathbb{R}|&\varphi\in C^{2,1,1}\!(W^{\Xi}\!\cap\!\{(x,t,y)|t<T\})\!\bigcap\! C^{1,0,1}\!((W^{\Xi})^{c}\!\cap\!\{(x,t,y)|t<T\})\!\bigcap\! C(\mathcal{Q}),\\
&\varphi_{x}\ \text{is absolutely continuous in}\ x,\\
&\mathbb{E}_{x,t,y}\int_{t}^{T}\Big[e^{-\beta(r-t)}\varphi_{x}(X_{r}^{\xi},r,\xi_{r})\sigma(X_{r}^{\xi},r,\xi_{r})\Big]^{2}dr<\infty,\ \  \forall(x,t,y)\in\mathcal{Q},\ \forall \xi\in\mathcal{D}_{[t, T]}\}.
\end{aligned}
\end{equation*}
Further, the function set $\mathcal{R}^{*}$ is defined by
\begin{equation*}
\begin{aligned}
\mathcal{R}^{*}:=\{\varphi:\mathcal{Q}\rightarrow\mathbb{R}|&\varphi\in C^{2,1,1}(\mathcal{Q}\cap\{(x,t,y)|t<T\})\bigcap C(\mathcal{Q}),\\
&\mathbb{E}_{x,t,y}\int_{t}^{T}\Big[e^{-\beta(r-t)}\varphi_{x}(X_{r}^{\xi},r,\xi_{r})\sigma(X_{r}^{\xi},r,\xi_{r})\Big]^{2}dr<\infty,\ \  \forall(x,t,y)\in\mathcal{Q},\ \forall \xi\in\mathcal{D}_{[t, T]}\}.
\end{aligned}
\end{equation*}
Clearly, it holds that $\mathcal{R}^{*}\subset\mathcal{R}^{\Xi}$ for any admissible $\Xi$.
\end{definition}

The next result is the verification theorem for the optimal singular control. It states that if the extended HJB equations have a solution with sufficient regularity, then there exists an optimal singular control that is generated by a singular control law. Later, in the RL design, to learn an optimal singular control, it suffices to learn the singular control law that simplifies the policy iterations in the learning procedure.

\begin{theorem}[Verification theorem for optimal singular control]
\label{verif}
Given a function $V\in\mathcal{R}^{*}$, if $\hat{\Xi}=(W^{\hat{\Xi}},(W^{\hat{\Xi}})^{c})$ defined by
\begin{equation*}
W^{\hat{\Xi}}:=\big\{(x,t,y)\in\mathcal{Q}\big|c(x,t,y)-V_{x}(x,t,y)+V_{y}(x,t,y)>0\big\}
\end{equation*}
is an admissible singular control law, and $V$ satisfies 
{\setlength{\abovedisplayskip}{2pt}%
 \setlength{\belowdisplayskip}{2pt}%
\begin{align}
&\!\min\!\bigg\{\!H(x,t,y)\!-\!\beta V(x,t,y)\!+\!V_{t}(x,t,y)\!+\!\mu(x,t,y)V_{x}(x,t,y)\!+\!\frac{1}{2}\sigma^{2}(x,t,y)V_{xx}(x,t,y)\!,\!\label{hjb1}\\
&\qquad\  c(x,t,y)-V_{x}(x,t,y)+V_{y}(x,t,y)\bigg\}=0,\quad \forall (x,t,y)\in\mathcal{Q},\ t<T,\notag\\
&V(x,T,y)=\tilde{F}(x,y),\quad \forall (x,y)\in\mathbb{R}\times[0,+\infty),\label{hjb2}
\end{align}
}
where $\tilde{F}(x,y):=\inf\limits_{a\ge 0}\big\{F(x-a,y+a)+\int_{0}^{a}c(x-u,T,y+u)du\big\}$, then the singular control generated by $\hat{\Xi}$ at $(x,t,y)$, denoted by $\xi^{x,t,y,\hat{\Xi}}$, is an optimal singular control for Problem (\ref{object}) starting at $(x,t,y)$. Moreover, $V(x,t,y)$ is the resulting optimal value function for Problem (\ref{object}) starting at $(x,t,y)$.
\end{theorem}

\begin{proof}
Fix $(x,t,y)\in\mathcal{Q}$ and denote  $\hat{\xi}:=\xi^{x,t,y,\hat{\Xi}}$,    $X^{\hat{\xi}}:=X^{x,t,y,\hat{\Xi}}$ for notational simplicity. The proof boils down to showing that for any admissible singular control $\xi$, we have $V(x,t,y)\le J(x,t,y;\xi)$ and that the equality holds for $\xi=\hat{\xi}$. The conclusion holds trivially at $t=T$, and it is sufficient to consider $t<T$.

Applying It\^{o}'s formula and taking expectations, we have
\begin{equation*}
\begin{aligned}
&\mathbb{E}_{x,t,y}e^{-\beta(T-t)}V(X^{\xi}_{T},T,\xi_{T})-V(x,t,y)\\
&=\mathbb{E}_{x,t,y}\bigg\{\int_{t}^{T}e^{-\beta(r-t)}\Big[-\beta V(X^{\xi}_{r},r,\xi_{r})+V_{t}(X^{\xi}_{r},r,\xi_{r})+\mu(X^{\xi}_{r},r,\xi_{r})V_{x}(X^{\xi}_{r},r,\xi_{r})\\
&+\frac{1}{2}\sigma^{2}(X^{\xi}_{r},r,\xi_{r})V_{xx}(X^{\xi}_{r},r,\xi_{r})\Big]dt+\int_{t}^{T}e^{-\beta(r-t)}\Big[-V_{x}(X^{\xi}_{r},r,\xi_{r})+V_{y}(X^{\xi}_{r},r,\xi_{r})\Big]d\xi^{c}_{r}\\
&+\sum\limits_{r\in[t,T]}e^{-\beta(r-t)}\int_{0}^{\Delta\xi_{r}}\Big[-V_{x}(X^{\xi}_{r-}-u,r,\xi_{r-}+u)+V_{y}(X^{\xi}_{r-}-u,r,\xi_{r-}+u)\Big]du\\
&+\int_{t}^{T}\Big[e^{-\beta(r-t)}V_{x}(X^{\xi}_{r},r,\xi_{r})\sigma(X^{\xi}_{r},r,\xi_{r})\Big]dB_{r}\bigg\}.
\end{aligned}
\end{equation*}
Because the last term vanishes for $V\in\mathcal{R}^{*}$, we obtain
\begin{align*}
&V(x,t,y)\\
=&\mathbb{E}_{x,t,y}\bigg\{e^{\!-\!\beta(T\!-\!t)}\tilde{F}(X^{\xi}_{T},\xi_{T})\!-\!\int_{t}^{T}\!e^{\!-\!\beta(r\!-\!t)}\big[\!-\!\beta V(X^{\xi}_{r},r,\xi_{r})\!+\!V_{t}(X^{\xi}_{r},r,\xi_{r})\!+\!\mu(X^{\xi}_{r},r,\xi_{r})V_{x}(X^{\xi}_{r},r,\xi_{r})\!\\
&+\frac{1}{2}\sigma^{2}(X^{\xi}_{r},r,\xi_{r})V_{xx}(X^{\xi}_{r},r,\xi_{r})\big]dr-\int_{t}^{T}e^{-\beta(r-t)}\Big[-V_{x}(X^{\xi}_{r},r,\xi_{r})+V_{y}(X^{\xi}_{r},r,\xi_{r})\Big]\circ d\xi_{r}\bigg\}\\
\le &\mathbb{E}_{x,t,y}\bigg\{e^{-\beta(T-t)}F(X^{\xi}_{T},\xi_{T})+\int_{t}^{T}e^{-\beta(r-t)}H(X^{\xi}_{r},r,\xi_{r})dr+\int_{t}^{T}e^{-\beta(r-t)}c(X^{\xi}_{r},r,\xi_{r})\circ d\xi_{r}\bigg\}\\
=&J(x,t,y;\xi).
\end{align*}
The above inequality holds in view that, for any $(x,t,y)\in\mathcal{Q}$,
\begin{equation*}
\begin{aligned}
&\tilde{F}(x,y)\le F(x,y),\\
&H(x,t,y)-\beta V(x,t,y)+V_{t}(x,t,y)+\mu(x,t,y)V_{x}(x,t,y)+\frac{1}{2}\sigma^{2}(x,t,y)V_{xx}(x,t,y)\ge 0,\\
&c(x,t,y)-V_{x}(x,t,y)+V_{y}(x,t,y)\ge 0.
\end{aligned}
\end{equation*}
Moreover, the inequality becomes equality for $\xi=\hat{\xi}$ because
\begin{equation*}
\begin{aligned}
&\tilde{F}(x,y)= F(x,y),\quad \forall (x,y) \  \text{s.t.} \  (x,T,y)\in W^{\hat{\Xi}},\\
&H(x,t,y)-\beta V(x,t,y)+V_{t}(x,t,y)+\mu(x,t,y)V_{x}(x,t,y)+\frac{1}{2}\sigma^{2}(x,t,y)V_{xx}(x,t,y)=0,\\
&\forall (x,t,y)\in W^{\hat{\Xi}},\  t<T,\\
&c(x,t,y)-V_{x}(x,t,y)+V_{y}(x,t,y)=0,\quad \forall  (x,t,y)\in (W^{\hat{\Xi}})^{c},\ t<T,
\end{aligned}
\end{equation*}
which completes the proof.
\end{proof}

It is expected that the regularity conditions in $\mathcal{R}^{*}$ may not hold for the value function of an arbitrary admissible control law. We need to use the It\^{o}-Tanaka-Meyer's formula to obtain a counterpart of the variational equation with $V$ replaced by $J(\cdot,\cdot,\cdot;\Xi)$ for an admissible control law $\Xi$ under the weakened smoothness condition in $\mathcal{R}^{\Xi}$.

\begin{theorem}
\label{verif-a}
Let $\Xi$ be an admissible singular control law. If a function $\tilde{J}(\cdot,\cdot,\cdot)\in\mathcal{R}^{\Xi}$ satisfies 
\begin{equation}
\label{eqad}
\begin{aligned}
&H(x,t,y)-\beta \tilde{J}(x,t,y)+\tilde{J}_{t}(x,t,y)+\mu(x,t,y)\tilde{J}_{x}(x,t,y)\\
&+\frac{1}{2}\sigma^{2}(x,t,y)\tilde{J}_{xx}(x,t,y)=0,\quad\quad\forall(x,t,y)\in W^{\Xi},\\
&c(x,t,y)-\tilde{J}_{x}(x,t,y)+\tilde{J}_{y}(x,t,y)=0,\quad\quad \forall(x,t,y)\in (W^{\Xi})^{c},\\
&\tilde{J}(x,T,y)=\bar{F}(x,y;\Xi),
\end{aligned}
\end{equation}
where $\bar{F}(x,y;\Xi):=F(x-D(x,y;\Xi),y+D(x,y;\Xi))+\int_{0}^{D(x,y;\Xi)}c(x-u,T,y+u)du$ with $D(x,y;\Xi):=\inf\big\{a\ge 0\big|(x-a,T,y+a)\in W^{\Xi}\big\}$, then $\tilde{J}$ is the value function of  $\Xi$, i.e., $\tilde{J}(x,t,y)=J(x,t,y;\Xi)$ for any $(x,t,y)\in\mathcal{Q}$.
\end{theorem}
\begin{proof}
The proof boils down to showing that the counterpart of the variational equation with $V$ replaced by $\tilde{J}(\cdot,\cdot,\cdot)$ still holds. Recall that the variational equation is deduced by applying It\^{o}'s formula to the continuous difference of $V(X^{\Xi}_{r},r,\xi^{\Xi})$. Under the weakened regularity condition in $\mathcal{R}^{\Xi}$, we can instead apply the It\^{o}-Tanaka-Meyer's formula to obtain 
\begin{align*}
&\tilde{J}(X^{\Xi}_{T-},T,\xi^{\Xi}_{T-})-\tilde{J}(x,t,y)\\
=&\sum\limits_{r\in[t,T)}\Big[\tilde{J}(X^{\Xi}_{r},r,\xi^{\Xi}_{r})-\tilde{J}(X^{\Xi}_{r-},r,\xi^{\Xi}_{r-})\Big]+\int_{t}^{T}\Big[-\beta \tilde{J}(X^{\Xi}_{r},r,\xi^{\Xi}_{r})+\tilde{J}_{t}(X^{\Xi}_{r},r,\xi^{\Xi}_{r})\\&+\mu(X^{\Xi}_{r},r,\xi^{\Xi}_{r})\tilde{J}_{x}(X^{\Xi}_{r},r,\xi^{\Xi}_{r})\Big]dr+\int_{t}^{T}\sigma (X^{\Xi}_{r},r,\xi^{\Xi}_{r})\tilde{J}_{x}(X^{\Xi}_{r},r,\xi^{\Xi}_{r})dB_{r}\\
&+\int_{t}^{T}\Big[-\tilde{J}_{x}(X^{\Xi}_{r},r,\xi^{\Xi}_{r})+\tilde{J}_{y}(X^{\Xi}_{r},r,\xi^{\Xi}_{r})\Big]d(\xi^{\Xi})^{c}_{r}+\int_{-\infty}^{\infty}\Lambda_{T}(a)\tilde{\mu}(da)\\
=&\sum\limits_{r\in[t,T)}\int_{0}^{\Delta\xi^{\Xi}_{r}}\Big[-\tilde{J}_{x}(X^{\Xi}_{r}-u,r,\xi^{\Xi}_{r}+u)+\tilde{J}_{y}(X^{\Xi}_{r}-u,r,\xi^{\Xi}_{r}+u)\Big]du\\&+\int_{t}^{T}\Big[-\beta \tilde{J}(X^{\Xi}_{r},r,\xi^{\Xi}_{r})+\tilde{J}_{t}(X^{\Xi}_{r},r,\xi^{\Xi}_{r})+\mu(X^{\Xi}_{r},r,\xi^{\Xi}_{r})\tilde{J}_{x}(X^{\Xi}_{r},r,\xi^{\Xi}_{r})\\&+\frac{1}{2}\sigma^{2}(X^{\Xi}_{r},r,\xi^{\Xi}_{r})\tilde{J}_{xx}(X^{\Xi}_{r},r,\xi^{\Xi}_{r})\Big]dr+\int_{t}^{T}\sigma (X^{\Xi}_{r},r,\xi^{\Xi}_{r})\tilde{J}_{x}(X^{\Xi}_{r},r,\xi^{\Xi}_{r})dB_{r}\\
&+\int_{t}^{T}\Big[-\tilde{J}_{x}(X^{\Xi}_{r},r,\xi^{\Xi}_{r})+\tilde{J}_{y}(X^{\Xi}_{r},r,\xi^{\Xi}_{r})\Big]d(\xi^{\Xi})^{c}_{r},\ \text{a.s.}, 
\end{align*}
where $\Lambda$ is the local time of $\big\{(X^{\Xi})^{c}_{r}\big\}_{r\in[t,T]}$ and $\tilde{\mu}$ is the second derivative measure of $\tilde{J}(\cdot,\cdot,\cdot)$; see Sect. 3.7 of \cite{karatzas1998brownian}. Then, taking expectations on both sides of the above equation yields the conclusion.
\end{proof}

\begin{remark}
The $C^{2,1,1}\!(W^{\Xi}\!\cap\!\{(x,t,y)|t<T\})\!\bigcap\! C^{1,0,1}\!((W^{\Xi})^{c}\!\cap\!\{(x,t,y)|t<T\})\!\bigcap\! C\!(\mathcal{Q})$ regularity with absolutely continuous first-order derivative, as required in $\mathcal{R}^{\Xi}$, is a natural regularity condition of the value function. Specifically, the value function usually has at most the first-order smooth fit at the boundary; see our example in Subsection \ref{4case}. For the problem that admits an explicit solution to (\ref{eqad}), one can verify the regularity of the value function as the solution is smooth in the waiting region and the action region with a first-order smooth fit across the boundary. For a more general problem, one may verify the regularity of the value function under some model assumptions; see our next result as an example.
\end{remark}

\begin{proposition}
If there exist $\alpha\in(0,1), C>0$ and $\kappa(x,t)$ in $L^{1}((-\infty,0])$ for each fixed $t$, such that the waiting region of $\Xi$ is given by $W^{\Xi}=\{(x,t,y)|x<\Gamma(t)\}$ for some $\Gamma\in C^{1+\frac{\alpha}{2}}([0,T])$, and 
\begin{align*}
&\mu(x,t,y)=\mu(x,t)\in C^{1+\alpha,\frac{\alpha}{2}}(\mathbb{R}\!\times\![0,T]),\quad  \sigma(x,t,y)=\sigma(x,t)\in C^{1+\alpha,\frac{\alpha}{2}}(\mathbb{R}\!\times\![0,T]),\quad \sigma(x,t,y)>0,\\
&H(x,t,y)=H(x,t)\in C^{1+\alpha,\frac{\alpha}{2}}(\mathbb{R}\!\times\![0,T]),\quad
 c(x,t,y)=c(t)\in C^{1+\frac{\alpha}{2}}([0,T]),\\
&F(x,y)=F(x)\in C^{3+\alpha}(\mathbb{R})\ \text{with}\ \lim\limits_{x\rightarrow-\infty}F'(x)=0,\\
&c(t),F(x)\in[0,C\kappa(x,t)],\quad
\mu_{x}(x,t)<\beta,\quad
 H_{x}(x,t)>0,\\
&\frac{1}{C}H_{x}(x+\Gamma(T-t),T-t)+[\mu_{x}(x+\Gamma(T-t),T-t)-\beta] \kappa(x,t)-\kappa_{t}(x,t)\\&+[\mu(x+\Gamma(T-t),T-t)+\sigma(x+\Gamma(T-t),T-t)\sigma_{x}(x+\Gamma(T-t),T-t)-\Gamma'(T-t)]\kappa_{x}(x,t)\\
&+\frac{1}{2}\sigma^{2}(x+\Gamma(T-t),T-t)\kappa_{xx}(x,t)<0,\quad\forall x\in(-\infty,0),
\end{align*}
and assume that $\kappa(x,t),\sigma(x,t)$ have polynomial growth in $x$. Then the value function $J(x,t,y;\Xi)$ is in $C^{2,1,1}\!(W^{\Xi}\!\cap\!\{(x,t,y)|t<T\})\!\bigcap\! C^{1,0,1}\!((W^{\Xi})^{c}\!\cap\!\{(x,t,y)|t<T\})\!\bigcap\! C\!(\mathcal{Q})$.
\end{proposition}
\begin{proof}
By the independence of $y$, we have $J(x,t,y;\Xi)=J(x,t;\Xi)$. Denote $\mathcal{W}^{\Xi}=\{(x,t)\in\mathbb{R}\times[0,T)|x<\Gamma(t)\}$ and $(\mathcal{W}^{\Xi})^{c}=\{(x,t)\in\mathbb{R}\times[0,T)|x\ge\Gamma(t)\}$.

The equation for $v(x,t)=J_{x}(x,T-t;\Xi),(x,t)\in \mathbb{R}\times[0,T)$ is given by
\begin{equation*}
\begin{aligned}
&H_{x}(x,T-t)+[\mu_{x}(x,T-t)-\beta] v(x,t)-v_{t}(x,t)+[\mu(x,T-t)+\sigma(x,T-t)\sigma_{x}(x,T-t)]v_{x}(x,t)\\
&+\frac{1}{2}\sigma^{2}(x,T-t)v_{xx}(x,t)=0,\quad\forall(x,t)\in (-\infty,\Gamma(T-t))\times[0,T),\\
&c(T-t)-v(x,t)=0,\quad \forall(x,t)\in (\Gamma(T-t),+\infty)\times[0,T),\\
&v(x,0)=\bar{F}_{x}(x;\Xi),
\end{aligned}
\end{equation*}
where $\bar{F}(x;\Xi)=F(x)1_{x<\Gamma(T)}+\Big[F\big(\Gamma(T)\big)+c(T)\big[x-\Gamma(T)\big]\Big]1_{x\ge \Gamma(T)}$. We emphasize that $\bar{F}_{x}(x;\Xi)$ can be discontinuous at $x=\Gamma(T)$.

The equation for $g(x,t)=v(x+\Gamma(T-t),t),(x,t)\in \mathbb{R}\times[0,T)$ is given by 
\begin{equation*}
\begin{aligned}
&H_{x}(x+\Gamma(T-t),T-t)+[\mu_{x}(x+\Gamma(T-t),T-t)-\beta] g(x,t)-g_{t}(x,t)\\&+[\mu(x+\Gamma(T-t),T-t)+\sigma(x+\Gamma(T-t),T-t)\sigma_{x}(x+\Gamma(T-t),T-t)-\Gamma'(T-t)]g_{x}(x,t)\\
&+\frac{1}{2}\sigma^{2}(x+\Gamma(T-t),T-t)g_{xx}(x,t)=0,\quad\forall(x,t)\in (-\infty,0)\times[0,T),\\
&c(T-t)-g(x,t)=0,\quad \forall(x,t)\in [0,+\infty)\times[0,T),\\
&g(x,0)=\bar{F}_{x}(x+\Gamma(T);\Xi),
\end{aligned}
\end{equation*}

For fixed $n>0$, by the Schauder theory for the first initial-boundary value problem (Theorem 1.16 of \cite{wang2021nonlinear}), there exists a unique $g^{(n)}\in C^{2,1}((-n,0)\times[0, T))\times C([-n,0]\times[0, T])$ solving
\begin{equation}
\label{eqg}
\begin{aligned}
&H_{x}(x+\Gamma(T-t),T-t)+[\mu_{x}(x+\Gamma(T-t),T-t)-\beta] g^{(n)}(x,t)-g^{(n)}_{t}(x,t)\\&+[\mu(x+\Gamma(T-t),T-t)+\sigma(x+\Gamma(T-t),T-t)\sigma_{x}(x+\Gamma(T-t),T-t)-\Gamma'(T-t)]g^{(n)}_{x}(x,t)\\
&+\frac{1}{2}\sigma^{2}(x+\Gamma(T-t),T-t)g^{(n)}_{xx}(x,t)=0,\quad\forall(x,t)\in (-n,0)\times[0,T),\\
&g^{(n)}(0,t)=c(T-t),\quad g^{(n)}(-n,t)=F_{x}(-n+\Gamma(T);\Xi),\\
&g^{(n)}(x,0)=F^{(n)}_{x}(x+\Gamma(T);\Xi),\quad\forall x\in[-n,0]
\end{aligned}
\end{equation}
where $F^{(n)}_{x}$ is a smooth function in $C^{2+\alpha}([-n,0])$ such that $F^{(n)}_{x}=F_{x}$ in $[-n+\Gamma(T),-\frac{1}{n}+\Gamma(T)]$, $F^{(n)}_{x}(\Gamma(T))=c(T)$ and $F^{(n)}_{x}(x)$ remains between $c(T)$ and $F^{(n)}_{x}(-\frac{1}{n}+\Gamma(T))$ during $x\in[-\frac{1}{n}+\Gamma(T),\Gamma(T)]$. Moreover, it follows from the comparison principle that
\begin{equation*}
0\le g^{(n)}(x,t)\le C\kappa(x,t).
\end{equation*}

Take arbitrary sequences of numbers $a^{m}\downarrow-\infty, b^{m}\uparrow0,c^{m}\downarrow0,d^{m}\uparrow T$ and denote $K^{m}:=[a^{m},b^{m}]\times[c^{m},d^{m}]$. Then for any $n\ge -a^{m}$, we have $g^{(n)}\in C^{2,1}(K^{m})$. By the interior Schauder estimate (Theorem 1.19 of \cite{wang2021nonlinear}), we get that, for any $n\ge -a^{m+1}$,
\begin{align*}
\Vert g^{(n)}\Vert_{C^{2+\alpha,1+\frac{\alpha}{2}}(K^{m})}\le& \tilde{C}(\Vert H_{x}(x+\Gamma(T-t),T-t)\Vert_{C^{\alpha,\frac{\alpha}{2}}(K^{m+1})}+\Vert g^{(n)}\Vert_{L^{2}(K^{m+1})})\\
\le&\tilde{C}(\Vert H_{x}(x+\Gamma(T-t),T-t)\Vert_{C^{\alpha,\frac{\alpha}{2}}(K^{m+1})}+\Vert \kappa\Vert_{L^{2}(K^{m+1})})
\end{align*}
where $\tilde{C}$ is some constant depending on $T$, the uniform parabolic bounds of $\sigma(x,t)$ on $K^{m+1}$, the $C^{\alpha,\frac{\alpha}{2}}$ norm of coefficients and the distance between $K^{m}$ and $\partial_{p}K^{m+1}$, but independent of $n$. Then we can use the standard diagonal argument to find a subsequence $g^{(n)}$ that converges in $C^{2+\alpha,1+\frac{\alpha}{2}}(K^{m})$ for any $m$. Denote the limit of $g^{(n)}$ by $g$ and extend $g(x,t)=c(T-t)$ for $(x,t)\in[0,+\infty)\times[0,T)$, then $g$ is a $C^{2+\alpha,1+\frac{\alpha}{2}}((-\infty,0)\times(0,T))\bigcap C\big((\mathbb{R}\times[0,T])\setminus\{(0,0)\} \big)$ solution to (\ref{eqg}). Moreover, $g(x,t)\le C\kappa(x,t)$ is in $L^{1}((-\infty,0])$ for each fixed $t$. Define $J(x,t;\Xi):=\int_{-\infty}^{x}v(x',T-t)dx'=\int_{-\infty}^{x}g(x'-\Gamma(t),T-t)dx'$, then $J\in C^{3,1}(\mathcal{W}^{\Xi})\bigcap C^{\infty,0}((\mathcal{W}^{\Xi})^{c})\bigcap C^{1,0}\big((\mathbb{R}\times[0,T])\setminus\{(\Gamma(T),T)\}\big)$. Moreover, for any $t$,  $|J^{-}_{xx}(\Gamma(t),t)|<\infty$ while $J^{+}_{xx}(\Gamma(t),t)=0$. It holds that, for fixed $t$, the function $J_{x}(x,t;\Xi)$ is Lipschitz in the neighborhood of $x=\Gamma(t)$ and is thus absolutely continuous. By the verification theorem \ref{verif-a}, $J$ is the value function of $\Xi$, which completes the proof. 
\end{proof}

Comparing with the verification of an admissible singular control law, we require more restrictive sufficient conditions to verify the optimal singular control law. First, additional inequalities in (\ref{hjb1}) are required. Second, the smoothness of the candidate value function is mandated to be $C^{2,1,1}$ on the whole space $\mathcal{Q}\cap\{(x,t,y)|t<T\}$. Moreover, a stronger integrability condition is needed, as the integrability condition can be weakened from that in $\mathcal{R}^{\Xi}$ to
\begin{equation*}
\mathbb{E}_{x,t,y}\!\int_{t}^{T}\!\Big[e^{-\beta(r-t)}\varphi_{x}(X_{r}^{x,t,y,\Xi},r,\xi^{x,t,y,\Xi}_{r})\sigma(X_{r}^{x,t,y,\Xi},r,\xi^{x,t,y,\Xi}_{r})\Big]^{2}dr<\infty,\ \forall(x,t,y)\in\mathcal{Q}
\end{equation*}
for the proof of Theorem \ref{verif-a}.

In the following sections, we are interested in the situation when the model coefficients $\mu,\sigma,H,F,c,\beta$ are unknown. The agent needs to learn the optimal singular control via the interactions with the environment. 

\section{RL Approach for Some Singular Control Problems}
\label{theory}
In this section, we introduce the theoretical foundations of our RL algorithms in a manner of deterministic policy iterations. In particular, we consider the formulation and policy iteration without the entropy regularization. The goal is to learn the value function $V$ satisfying the verification theorem (Theorem \ref{verif}), assuring its role as the desired optimal value function. Given the learned optimal value function $V$, we can construct the optimal singular control law $\hat{\Xi}$ that generates an optimal singular control as in Theorem \ref{verif}.

Recall that in our framework, the optimal singular control law $\hat{\Xi}$ rather than the optimal singular control has a feedback form. Therefore, we expect that it is more convenient and effective to implement the policy evaluation and iteration over the singular control law.

\subsection{Policy evaluation and martingale characterization}
This subsection focuses on the policy evaluation (PE) of the singular control laws.  We fix $(x,t,y)\in\mathcal{Q}$ and denote $(X^{\Xi},\xi^{\Xi}):=(X^{x,t,y,\Xi},\xi^{x,t,y,\Xi})$ for any admissible $\Xi$. The policy evaluation step is to approximate the value function of the singular control generated by $\Xi$,  i.e.,   $J(x,t,y;\Xi)$.  We have the following result that links the value function of $\Xi$ to a martingale process.

\begin{theorem}
\label{PET}
Let $\Xi$ be an admissible singular control law and $\tilde{J}(\cdot,\cdot,\cdot)$ be a given function. Then the following statements (1)-(3) are equivalent:\\
(1) $\tilde{J}(\cdot,\cdot,\cdot)$ equals the value function of $\Xi$, i.e., $\tilde{J}(x,t,y)=J(x,t,y;\Xi)$ for any $(x,t,y)\in\mathcal{Q}$.\\
(2) $\tilde{J}(x,T,y)=\bar{F}(x,y;\Xi)$ and the process
\begin{equation*}
\begin{aligned}
M=\bigg\{M_{r}&:=&e^{-\beta(r-t)}\tilde{J}(X^{\Xi}_{r},r,\xi^{\Xi}_{r})+\int_{t}^{r}e^{-\beta(r'-t)}H(X^{\Xi}_{r'},r',\xi^{\Xi}_{r'})dr'\\
&& +\int_{t}^{r}e^{-\beta(r'-t)}c(X^{\Xi}_{r'},r',\xi^{\Xi}_{r'})\circ d\xi^{\Xi}_{r'},\ r\in[t,T]\bigg \} 
\end{aligned}
\end{equation*}
is an $\{\mathcal{F}_{r}\}_{r\in [t, T]}$-martingale.\\
(3) $\tilde{J}(x,T,y)=\bar{F}(x,y;\Xi)$ and the process
\begin{equation*}
\begin{aligned}
M_{.-}=\bigg\{M_{r-}\!\!&:=&\!\!e^{-\beta(r-t)}\tilde{J}(X^{\Xi}_{r-},r,\xi^{\Xi}_{r-})+\int_{t}^{r}e^{-\beta(r'-t)}H(X^{\Xi}_{r'},r',\xi^{\Xi}_{r'})dr'\\
&&+\int_{t}^{r-}e^{-\beta(r'-t)}c(X^{\Xi}_{r'},r',\xi^{\Xi}_{r'})\circ d\xi^{\Xi}_{r'},\ r\in[t,T]\bigg\}
\end{aligned}
\end{equation*}
 is an $\{\mathcal{F}_{r-}\}_{r\in[t,T]}$-martingale, where the last term is defined by
\begin{equation*}
\int_{t}^{r-}\!\!\!e^{-\beta(r'-t)}c(X^{\Xi}_{r'},r',\xi^{\Xi}_{r'})\circ d\xi^{\Xi}_{r'}:=\int_{t}^{r}\!\!e^{-\beta(r'-t)}c(X^{\Xi}_{r'},r',\xi^{\Xi}_{r'})d(\xi^{\Xi})^{c}_{r'}\!+\!\!\!\sum\limits_{r'\in[t,r)}\int_{0}^{\Delta\xi^{\Xi}_{r'}}\!\!c(X^{\Xi}_{r'-}\!-\!u,r',\xi^{\Xi}_{r'-}\!+\!u)du.
\end{equation*}
\end{theorem}
\begin{proof}
The $\{M_{r}\}_{r\in[t,T]}$ in (2) and the $\{M_{r-}\}_{r\in[t,T]}$ in (3) satisfies $M_{r-}=\lim\limits_{r'\uparrow r}M_{r}$, which naturally implies $(2)\Rightarrow(3)$. We need to show $(3)\Rightarrow(1)$ and $(1)\Rightarrow(2)$.

For the direction (3)$\Rightarrow$(1), if $\tilde{J}(x,T,y)=\bar{F}(x,y;\Xi)$ and $M_{.-}$ is  an $\{\mathcal{F}_{r-}\}_{r\in[t,T]}$-martingale, then
\begin{equation*}
\mathbb{E}\Big[M_{T-}\Big|\mathcal{F}_{r-}\Big]=M_{r-},
\end{equation*}
and
\begin{equation*}
\begin{aligned}
M_{T}\!-\!M_{T-}\!=&e^{-\beta(T-t)}\tilde{J}(X^{\Xi}_{T},T,\xi^{\Xi}_{T})-e^{-\beta(T-t)}\tilde{J}(X^{\Xi}_{T-},T,\xi^{\Xi}_{T-})+e^{-\beta(T-t)}\int_{0}^{\Delta\xi^{\Xi}_{T}}c(X^{\Xi}_{T-}-u,T,\xi^{\Xi}_{T-}+u)du\\
\!=&0,
\end{aligned}
\end{equation*}
which implies
\begin{equation*}
\begin{aligned}
&e^{-\beta(r-t)}\tilde{J}(X^{\Xi}_{r-},r,\xi^{\Xi}_{r-})+\int_{t}^{r}e^{-\beta(r'-t)}H(X^{\Xi}_{r'},r',\xi^{\Xi}_{r'})dr'+\int_{t}^{r-}e^{-\beta(r'-t)}c(X^{\Xi}_{r'},r',\xi^{\Xi}_{r'})\circ d\xi^{\Xi}_{r'}\\
&=\mathbb{E}\Big[e^{-\beta(T-t)}F(X^{\Xi}_{T},\xi^{\Xi}_{T})+\int_{t}^{T}e^{-\beta(r'-t)}H(X^{\Xi}_{r'},r',\xi^{\Xi}_{r'})dr'+\int_{t}^{T}e^{-\beta(r'-t)}c(X^{\Xi}_{r'},r',\xi^{\Xi}_{r'})\circ d\xi^{\Xi}_{r'}\Big|\mathcal{F}_{r-}\Big].
\end{aligned}
\end{equation*}
Letting $r=t$ yields
\begin{equation*}
\tilde{J}(x,t,y)=\mathbb{E}_{x,t,y}\bigg[\int_{t}^{T}e^{-\beta(r'-t)}H(X^{\Xi}_{r'},r',\xi^{\Xi}_{r'})dr'+e^{-\beta(T-t)}F(X^{\Xi}_{T},\xi^{\Xi}_{T})+\int_{t}^{T}e^{-\beta(r'-t)}c(X^{\Xi}_{r'},r',\xi^{\Xi}_{r'})\circ d\xi^{\Xi}_{r'}\bigg],
\end{equation*}
which coincides with the definition of $J(\cdot,\cdot,\cdot;\Xi)$.

For the direction (1)$\Rightarrow$(2), if $\tilde{J}(\cdot,\cdot,\cdot)$ equals the value function of $\Xi$, then the terminal condition $\tilde{J}(x,T,y)=\bar{F}(x,y;\Xi)$ clearly holds. It can be verified by the definition of $J(\cdot,\cdot,\cdot;\Xi)$ that
\begin{equation*}
\mathbb{E}\Big[M_{T}\Big|\mathcal{F}_{r}\Big]=M_{r},
\end{equation*}
hence $\{M_{r}\}_{r\in[t,T]}$ is an $\{\mathcal{F}_{r}\}_{r\in[t,T]}$-martingale.
\end{proof}

Based on the above martingale characterization, we can apply the method in \cite{jia2022policy} to design some PE algorithms for our singular control problems; see Section \ref{algorithm}. 

\subsection{Region iteration and the policy improvement result}
In the last subsection, we develop a PE theory to learn the value function of any admissible singular control law. This subsection focuses on the discussion on the iteration of regions such that we can update the singular control laws to approximate the optimal singular control law.

We first have the policy improvement (PI) result.
\begin{theorem}
\label{PIT}
Under Assumptions \ref{as1}, let $\Xi$ be an admissible singular control law with its value function $J(x,t,y;\Xi)\in C^{2,1,1}\!(W^{\Xi}\!\cap\!\{(x,t,y)|t<T\})\!\bigcap\! C^{1,0,1}\!((W^{\Xi})^{c}\!\cap\!\{(x,t,y)|t<T\})\!\bigcap\! C\!(\mathcal{Q})$ with $J_{x}(x,t,y;\Xi)$ absolutely continuous in $x$. Define $\tilde{\Xi}=(W^{\tilde{\Xi}},(W^{\tilde{\Xi}})^{c})$ with $W^{\tilde{\Xi}}:=\bigcup\limits_{r\in[0,T]}W^{\tilde{\Xi}}_{r}$, where $W^{\tilde{\Xi}}_{r}$ is the two dimensional slices at time $r$ given by
\begin{equation*}
\begin{aligned}
&W^{\tilde{\Xi}}_{t}:=int\Big(\big\{(x,t,y)\big|-\beta J(x,t,y;\Xi)+J_{t}(x,t,y;\Xi)+\mu(x,t,y) J_{x}(x,t,y;\Xi)\\
&\!+\!\frac{1}{2}\!\sigma^{2}\!(\!x\!,\!t\!,\!y\!)J_{xx}\!(x,t,y;\Xi))\!+\!H\!(x,t,y)\!\le\! 0,c\!(\!x\!,\!t\!,\!y\!)\!-\!J_{x}\!(x,t,y;\Xi)\!+\!J_{y}\!(x,t,y;\Xi)\!\ge\! 0\!\big\}\!\Big)\!,\ \forall t\!< \!T\! ,\!\\
&W^{\tilde{\Xi}}_{T}:=\big\{(x,T,y)\big|c(x,T,y)-J_{x}(x,T,y;\Xi)+J_{y}(x,T,y;\Xi)>0\big\},
\end{aligned}
\end{equation*}
where $int(A)$ stands for the interior of the set $A$. If $\tilde{\Xi}$ is admissible, then
$$J(x,t,y;\tilde{\Xi})\le J(x,t,y;\Xi).$$
Moreover, if the mapping $\mathcal{I}:\Xi\mapsto\tilde{\Xi}=\mathcal{I}(\Xi)$ has a fixed point $\hat{\Xi}$ that is admissible such that its value function is in $\mathcal{R}^{*}$, we have that $\hat{\Xi}$ is an optimal strategy and $J(x,t,y;\hat{\Xi})$ is the optimal value function.
\end{theorem}

\begin{proof}
Let $(X^{\Xi},\xi^{\Xi})$ and  $(X^{\tilde{\Xi}},\xi^{\tilde{\Xi}})$ be the solutions of state-control pairs to the Skorokhod reflection problem w.r.t strategy $\Xi$ and $\tilde{\Xi}$, respectively. 

Applying the It\^{o}-Tanaka-Meyer's formula under the regularity assumption of $J(\cdot,\cdot,\cdot;\Xi)$, we have
\begin{align*}
&e^{-\beta (T-t)}J(X^{\tilde{\Xi}}_{T-},T,\xi^{\tilde{\Xi}}_{T-};\Xi)-J(x,t,y;\Xi)+\int_{t}^{T}e^{-\beta(r-t)}H(X^{\tilde{\Xi}}_{r},r,\xi^{\tilde{\Xi}}_{r})dr+\int_{t}^{T-}e^{-\beta(r-t)}c(X^{\tilde{\Xi}}_{r},r,\xi^{\tilde{\Xi}}_{r})\circ d\xi^{\tilde{\Xi}}_{r}\\
=&\int_{t}^{T}e^{-\beta(r-t)}\Big[H(X^{\tilde{\Xi}}_{r},r,\xi^{\tilde{\Xi}}_{r})-\beta J(X^{\tilde{\Xi}}_{r},r,\xi^{\tilde{\Xi}}_{r};\Xi)+J_{t}(X^{\tilde{\Xi}}_{r},r,\xi^{\tilde{\Xi}}_{r};\Xi)+\mu(X^{\tilde{\Xi}}_{r},r,\xi^{\tilde{\Xi}}_{r}) J_{x}(X^{\tilde{\Xi}}_{r},r,\xi^{\tilde{\Xi}}_{r};\Xi)\\&+\frac{1}{2}\sigma^{2}(X^{\tilde{\Xi}}_{r},r,\xi^{\tilde{\Xi}}_{r})J_{xx}(X^{\tilde{\Xi}}_{r},r,\xi^{\tilde{\Xi}}_{r};\Xi)\Big]dr+\int_{t}^{T}e^{-\beta(r-t)}\Big[\sigma(X^{\tilde{\Xi}}_{r},r,\xi^{\tilde{\Xi}}_{r})J_{x}(X^{\tilde{\Xi}}_{r},r,\xi^{\tilde{\Xi}}_{r};\Xi)\Big]dB_{r}\\
&+\int_{t}^{T}e^{-\beta(r-t)}\Big[c(X^{\tilde{\Xi}}_{r},r,\xi^{\tilde{\Xi}}_{r})-J_{x}(X^{\tilde{\Xi}}_{r},r,\xi^{\tilde{\Xi}}_{r};\Xi)+J_{y}(X^{\tilde{\Xi}}_{r},r,\xi^{\tilde{\Xi}}_{r};\Xi)\Big]d(\xi^{\tilde{\Xi}})^{c}_{r}\\
&+\sum\limits_{r\in[t,T)}e^{-\beta(r-t)}\int_{0}^{\Delta\xi^{\tilde{\Xi}}_{r}}\Big[c(X^{\tilde{\Xi}}_{r-}\!-\!u,r,\xi^{\tilde{\Xi}}_{r-}\!+\!u)-J_{x}(X^{\tilde{\Xi}}_{r-}\!-\!u,r,\xi^{\tilde{\Xi}}_{r-}\!+\!u;\Xi)+J_{y}(X^{\tilde{\Xi}}_{r-}\!-\!u,r,\xi^{\tilde{\Xi}}_{r-}\!+\!u;\Xi)\Big]du.
\end{align*}

In the above equality, plugging in $J(X^{\tilde{\Xi}}_{T-},T,\xi^{\tilde{\Xi}}_{T-};\Xi)=F(X^{\tilde{\Xi}}_{T-},\xi^{\tilde{\Xi}}_{T-})$ and replacing $T$ by $T\wedge\tau_{n}$ with $\tau_{n}=\inf\big\{r\ge t\big||X^{\tilde{\Xi}}_{r}+\xi^{\tilde{\Xi}}_{r}|\ge n\big\}$, we deduce that $\sigma(X^{\tilde{\Xi}}_{r},r,\xi^{\tilde{\Xi}}_{r})$ and $J_{x}(X^{\tilde{\Xi}}_{r},r,\xi^{\tilde{\Xi}}_{r};\Xi),r\in[t,T\wedge\tau_{n}]$, are bounded under Assumptions \ref{as1} and that $J_{x}(x,t,y;\Xi)$ is absolutely continuous in $x$. Then taking expectations on both sides of the above equality yields
\begin{equation}
\label{PF}
\begin{aligned}
&\mathbb{E}_{x,t,y}\bigg\{e^{-\beta (T\wedge\tau_{n}-t)}F(X^{\tilde{\Xi}}_{T\wedge\tau_{n}-},\xi^{\tilde{\Xi}}_{T\wedge\tau_{n}-})-J(x,t,y;\Xi)+
\int_{t}^{T\wedge\tau_{n}}e^{-\beta(r-t)}H(X^{\tilde{\Xi}}_{r},r,\xi^{\tilde{\Xi}}_{r})dr\\&+\int_{t}^{T\wedge\tau_{n}-}e^{-\beta(r-t)}c(X^{\tilde{\Xi}}_{r},r,\xi^{\tilde{\Xi}}_{r})\circ d\xi^{\tilde{\Xi}}_{r}\bigg\}\\
=&\mathbb{E}_{x,t,y}\bigg\{\int_{t}^{T\wedge\tau_{n}}e^{-\beta(r-t)}\Big[H(X^{\tilde{\Xi}}_{r},r,\xi^{\tilde{\Xi}}_{r})-\beta J(X^{\tilde{\Xi}}_{r},r,\xi^{\tilde{\Xi}}_{r};\Xi)+J_{t}(X^{\tilde{\Xi}}_{r},r,\xi^{\tilde{\Xi}}_{r};\Xi)\\
&+\mu(X^{\tilde{\Xi}}_{r},r,\xi^{\tilde{\Xi}}_{r}) J_{x}(X^{\tilde{\Xi}}_{r},r,\xi^{\tilde{\Xi}}_{r};\Xi)+\frac{1}{2}\sigma^{2}(X^{\tilde{\Xi}}_{r},r,\xi^{\tilde{\Xi}}_{r})J_{xx}(X^{\tilde{\Xi}}_{r},r,\xi^{\tilde{\Xi}}_{r};\Xi)\Big]dr\\
&+\int_{t}^{T\wedge\tau_{n}}e^{-\beta(r-t)}\Big[c(X^{\tilde{\Xi}}_{r},r,\xi^{\tilde{\Xi}}_{r})-J_{x}(X^{\tilde{\Xi}}_{r},r,\xi^{\tilde{\Xi}}_{r};\Xi)+J_{y}(X^{\tilde{\Xi}}_{r},r,\xi^{\tilde{\Xi}}_{r};\Xi)\Big]\circ d\xi^{\tilde{\Xi}}_{r}\bigg\}.
\end{aligned}
\end{equation}
Based on the definition of $\tilde{\Xi}$, the first term on the right-hand-side of Eq.(\ref{PF}) is non-positive because \\ $(X^{\tilde{\Xi}}_{r},r,\xi^{\tilde{\Xi}}_{r})\in\overline{ W^{\tilde{\Xi}}}$ for any $r\in[0,T]$. Under the regularity assumption, one can show that the value function $J(\cdot,\cdot,\cdot;\Xi)$ satisfies the HJB variational equality in Theorem \ref{verif-a}, which yields
\begin{equation*}
\begin{aligned}
\!\Big\{\!(x,t,y)\!\Big|&\!-\!\beta J\!(x,t,y;\Xi)\!+\!J_{t}\!(x,t,y;\Xi)\!+\!\mu\!(x,t,y) J_{x}\!(x,t,y;\Xi)\!+\!\frac{1}{2}\sigma^{2}\!(x,t,y)J_{xx}\!(x,t,y;\Xi)\!\\
&+H(x,t,y)>0,c(x,t,y)-J_{x}(x,t,y;\Xi)+J_{y}(x,t,y;\Xi)> 0\Big\}=\emptyset.
\end{aligned}
\end{equation*}
and  $(W^{\tilde{\Xi}})^{c}\subset\{(x,t,y)|c(x,t,y)-J_{x}(x,t,y;\Xi)+J_{y}(x,t,y;\Xi)\le 0\}$.  As such, the integrand in the second term of  Eq.(\ref{PF}) is zero for $(X^{\tilde{\Xi}}_{r-},r,\xi^{\tilde{\Xi}}_{r-})\in \overline{W^{\tilde{\Xi}}}$, and are non-positive for $(X^{\tilde{\Xi}}_{r-},r,\xi^{\tilde{\Xi}}_{r-})\in (W^{\tilde{\Xi}})^{c}$. Thus,
\begin{equation}
\label{eq1}
\begin{aligned}
&J(x,t,y;\Xi)
\ge \mathbb{E}_{x,t,y}\Bigg\{e^{-\beta(T\wedge\tau_{n}-t)}F(X^{\tilde{\Xi}}_{T\wedge\tau_{n}-},\xi^{\tilde{\Xi}}_{T\wedge\tau_{n}-})\\ &+\int_{t}^{T\wedge\tau_{n}}e^{-\beta (r-t)}H(X^{\tilde{\Xi}}_{r},r,\xi^{\tilde{\Xi}}_{r})dr+\int_{t}^{T\wedge\tau_{n}-}e^{-\beta (r-t)}c(X^{\tilde{\Xi}}_{r},r,\xi^{\tilde{\Xi}}_{r})\circ d\xi^{\tilde{\Xi}}_{r}\Bigg\}.
\end{aligned}
\end{equation}
Noting that 
\begin{align*}
\mathbb{E}_{x,t,y}
&e^{-\beta(T\wedge\tau_{n}-t)}F(X^{\tilde{\Xi}}_{T\wedge\tau_{n}},\xi^{\tilde{\Xi}}_{T\wedge\tau_{n}})\le \mathbb{E}_{x,t,y}\bigg\{F(X^{\tilde{\Xi}}_{T},\xi^{\tilde{\Xi}}_{T})1_{\tau_{n}\ge T}+F(X^{\tilde{\Xi}}_{\tau_{n}},\xi^{\tilde{\Xi}}_{\tau_{n}})1_{\tau_{n}\le T}\bigg\}\\
&\le \mathbb{E}_{x,t,y}\bigg\{C(1+|X^{\tilde{\Xi}}_{T}+\xi^{\tilde{\Xi}}_{T}|^{k})\bigg\}+\mathbb{E}_{x,t,y}\bigg\{C(1+n^{k})1_{\tau_{n}\le T})\bigg\}\\
&\le C(1+|x+y|^{k})+C(1+n^k)\mathbb{E}_{x,t,y}\bigg\{\frac{\max\limits_{r\in[t,T]}|X^{\tilde{\Xi}}_{r}+\xi^{\tilde{\Xi}}_{r}|^{k+1}}{n^{k+1}}\bigg\}\\
&\le C(1+|x+y|^{k})+C\frac{(1+n^k)(1+|x+y|^{k+1})}{n^{k+1}}
\end{align*}
is bounded by a constant independent of $n$. Then, letting $n\rightarrow\infty$ in (\ref{eq1}), applying the Dominated Convergence Theorem for the first term and Monotone Convergence Theorem for the remaining two terms, we obtain
\begin{equation*}
J(x,t,y;\Xi)\ge\mathbb{E}_{x,t,y}\bigg\{e^{-\beta(T-t)}F(X^{\tilde{\Xi}}_{T-},\xi^{\tilde{\Xi}}_{T-})+\int_{t}^{T}\!\!e^{-\beta (r-t)}H(X^{\tilde{\Xi}}_{r},r,\xi^{\tilde{\Xi}}_{r})dr+\int_{t}^{T-}\!\!\!\!e^{-\beta (r-t)}c(X^{\tilde{\Xi}}_{r},r,\xi^{\tilde{\Xi}}_{r})\circ d\xi^{\tilde{\Xi}}_{r}\bigg\}.
\end{equation*}
In view of the expression of $W^{\tilde{\Xi}}_{T}$, we have almost surely that
\begin{equation*}
F(X^{\tilde{\Xi}}_{T-},\xi^{\tilde{\Xi}}_{T-})\ge F(X^{\tilde{\Xi}}_{T},\xi^{\tilde{\Xi}}_{T})+\int_{0}^{\Delta\xi^{\tilde{\Xi}}_{T}}c(X^{\tilde{\Xi}}_{T-}-u,T,\xi^{\tilde{\Xi}}_{T-}+u)du.
\end{equation*}
Combining the previous two results leads to the desired inequality
\begin{equation*}
J(x,t,y;\tilde{\Xi})\leq J(x,t,y;\Xi). 
\end{equation*}
Suppose that $\hat{\Xi}$ is a fixed point of the map $\mathcal{I}$. Then 
\begin{equation*}
\begin{aligned}
&\!W^{\hat{\Xi}}_{t}= int\bigg(\Big\{(x,t,y)|\!-\!\beta J(x,t,y;\hat{\Xi})\!+\!J_{t}(x,t,y;\hat{\Xi})\!+\!\mu(x,t,y) J_{x}(x,t,y;\hat{\Xi})\\&\!+\!\frac{1}{2}\sigma^{2}(x,t,y)J_{xx}(x,t,y;\hat{\Xi})\!+\!H(x,t,y)\!\le\! 0,c(x,t,y)\!-\!J_{x}(x,t,y;\hat{\Xi})\!+\!J_{y}(x,t,y;\hat{\Xi})\!\ge\! 0\!\Big\}\!\bigg)\!,\!\\
&\!W^{\hat{\Xi}}_{T}=\bigg\{(x,t,y)|c(x,t,y)-J_{x}(x,t,y;\hat{\Xi})+J_{y}(x,t,y;\hat{\Xi})>0\bigg\}.
\end{aligned}
\end{equation*}
Then, as $J$ satisfies the variational equalities in Theorem \ref{verif-a}, we deduce that $J(x,t,y;\hat{\Xi})$ satisfies (\ref{hjb1}). In addition, the terminal condition (\ref{hjb2}) follows from the expression of $W^{\hat{\Xi}}_{T}$. Thanks to Theorem \ref{verif},  $\hat{\Xi}$ is indeed an optimal strategy.
\end{proof}

\begin{remark}
We note that the above theorem does not guarantee the existence of a fixed point of $\mathcal{I}$. In numerical implementations, the existence of the fixed point can be illustrated by iterating the mapping consecutively until convergence. From the theoretical perspective, the existence relates to the convergence of policy iteration, which is a topical and challenging problem in continuous-time RL; see \cite{huang2025convergence}, \cite{tran2025policy}, \cite{ma2025convergence} for some recent progress. In our concrete example in Subsection \ref{4case}, this convergence of policy iteration can be rigorously verified. 
\end{remark}

\subsection{q-functions and martingale characterization}
\label{fin-q}
Note that the previous iteration operator depends on unknown model parameters  $\mu$ and $\sigma$. 
In the practical application of the previous policy iteration, one needs to estimate the values of  $\mu$ and $\sigma$ in each iteration. In this subsection, we are interested in devising some model-free RL algorithms without the inefficient and inaccurate estimations of individual model parameters. In particular, we consider the following pair of q-functions that will be used to update the region of the singular control law. 
For a given admissible singular control law $\Xi$,  define
\begin{equation*}
\begin{aligned}
q_{0}(x,t,y;\Xi):=&c(x,t,y)-J_{x}(x,t,y;\Xi)+J_{y}(x,t,y;\Xi),\\
q_{1}(x,t,y;\Xi):=&-\beta J(x,t,y;\Xi)+J_{t}(x,t,y;\Xi)+\mu(x,t,y) J_{x}(x,t,y;\Xi)\\
&+\frac{1}{2}\sigma^{2}(x,t,y)J_{xx}(x,t,y;\Xi)+H(x,t,y).
\end{aligned}
\end{equation*}
Then we need to learn $q_{0}$ and $q_{1}$ to apply the PI theorem. The above functions can be interpreted as the zero-order and first-order q-functions, with the first-order q-function equaling the q-function for regular control; see \cite{jia2023q}. The zero-order q-function appears due to the variational inequality stemming from the singular control. 

We have the following martingale characterization results for singular control problems using the It\^{o}-Tanaka-Meyer's formula.

\begin{theorem}[Martingale characterization of the $(q_{0},q_{1})$ pair]
\label{qforq}
Given an admissible singular control law $\Xi$ with its value function $J(\cdot,\cdot,\cdot;\Xi)\in\mathcal{R}^{\Xi}$ and continuous functions $\tilde{q}_{0},\tilde{q}_{1}:\mathcal{Q}\rightarrow \mathbb{R}$, then\\
(1) 
\begin{equation*}
\begin{aligned}
&\tilde{q}_{0}(x,t,y)=q_{0}(x,t,y;\Xi),\quad \forall(x,t,y)\in (W^{\Xi})^{c},\\
&\tilde{q}_{1}(x,t,y)=q_{1}(x,t,y;\Xi),\quad \forall(x,t,y)\in W^{\Xi}
\end{aligned}
\end{equation*}
if and only if the process
\begin{equation*}
\begin{aligned}
&\bigg\{e^{-\beta(r-t)}J(X^{\Xi}_{r},r,\xi^{\Xi}_{r};\Xi)+\int_{t}^{r}e^{-\beta(r'-t)}\Big[H(X^{\Xi}_{r'},r',\xi^{\Xi}_{r'})-\tilde{q}_{1}(X^{\Xi}_{r'},r',\xi^{\Xi}_{r'})\Big]dr'\\
&+\int_{t}^{r}e^{-\beta(r'-t)}\Big[c(X^{\Xi}_{r'},r',\xi^{\Xi}_{r'})-\tilde{q}_{0}(X^{\Xi}_{r'},r',\xi^{\Xi}_{r'})\Big]\circ d\xi^{\Xi}_{r'},\  r\in[t,T]\bigg\}
\end{aligned}
\end{equation*}
is an $\{\mathcal{F}_{r}\}_{r\in[t,T]}$-martingale.\\
(2) The following three statements are equivalent:\\
(i)
\begin{equation*}
\begin{aligned}
&\tilde{q}_{0}(x,t,y)=q_{0}(x,t,y;\Xi),\quad \forall(x,t,y)\in \mathcal{Q},\\
&\tilde{q}_{1}(x,t,y)=q_{1}(x,t,y;\Xi),\quad \forall(x,t,y)\in \mathcal{Q}
\end{aligned}
\end{equation*}
(ii) For any admissible singular control $\xi'$, the process
\begin{equation*}
\begin{aligned}
&\bigg\{e^{-\beta(r-t)}J(X^{\xi'}_{r},r,\xi'_{r};\Xi)+\int_{t}^{r}e^{-\beta(r'-t)}\Big[H(X^{\xi'}_{r'},r',\xi'_{r'})-\tilde{q}_{1}(X^{\xi'}_{r'},r',\xi'_{r'})\Big]dr'\\
&+\int_{t}^{r}e^{-\beta(r'-t)}\Big[c(X^{\xi'}_{r'},r',\xi'_{r'})-\tilde{q}_{0}(X^{\xi'}_{r'},r',\xi'_{r'})\Big]\circ d\xi'_{r'},\  r\in[t,T]\bigg\}
\end{aligned}
\end{equation*}
is an $\{\mathcal{F}_{r}\}_{r\in[t,T]}$-martingale.\\
(iii) For any admissible singular control law $\Xi'$, the process
\begin{equation}
\label{m1_b}
\begin{aligned}
&\bigg\{e^{-\beta(r-t)}J(X^{\Xi'}_{r},r,\xi^{\Xi'}_{r};\Xi)+\int_{t}^{r}e^{-\beta(r'-t)}\Big[H(X^{\Xi'}_{r'},r',\xi^{\Xi'}_{r'})-\tilde{q}_{1}(X^{\Xi'}_{r'},r',\xi^{\Xi'}_{r'})\Big]dr'\\
&+\int_{t}^{r}e^{-\beta(r'-t)}\Big[c(X^{\Xi'}_{r'},r',\xi^{\Xi'}_{r'})-\tilde{q}_{0}(X^{\Xi'}_{r'},r',\xi^{\Xi'}_{r'})\Big]\circ d\xi^{\Xi'}_{r'},\  r\in[t,T]\bigg\}
\end{aligned}
\end{equation}
is an $\{\mathcal{F}_{r}\}_{r\in[t,T]}$-martingale.
\end{theorem}
\begin{proof}
First, we claim that for any admissible $\Xi'$, the following two statements are equivalent:\\
(a). It holds that
\begin{align}
&\tilde{q}_{0}(x,t,y)=q_{0}(x,t,y;\Xi),\quad \forall(x,t,y)\in (W^{\Xi'})^{c},\label{a1}\\
&\tilde{q}_{1}(x,t,y)=q_{1}(x,t,y;\Xi),\quad \forall(x,t,y)\in W^{\Xi'}.\label{a2}
\end{align}
(b). The process (\ref{m1_b}) is an $\{\mathcal{F}_{r}\}_{r\in[t,T]}$-martingale.

\textit{Proof of claim:} Applying It\^{o}-Tanaka-Meyer's formula to $\big\{e^{-\beta(r'-t)}J(X^{\Xi'}_{r'},r',\xi^{\Xi'}_{r'};\Xi)\big\}_{r'\in[t,r]}$ yields
\begin{equation*}
\begin{aligned}
e^{-\beta(r-t)}&J(X^{\Xi'}_{r},r,\xi^{\Xi'}_{r};\Xi)-J(x,t,y;\Xi)+\int_{t}^{r}e^{-\beta(r'-t)}\Big[H(X^{\Xi'}_{r'},r',\xi^{\Xi'}_{r'})-\tilde{q}_{1}(X^{\Xi'}_{r'},r',\xi^{\Xi'}_{r'})\Big]dr'\\&+\int_{t}^{r}e^{-\beta(r'-t)}\Big[c(X^{\Xi'}_{r'},r',\xi^{\Xi'}_{r'})-\tilde{q}_{0}(X^{\Xi'}_{r'},r',\xi^{\Xi'}_{r'})\Big]\circ d\xi^{\Xi'}_{r'}\\
=&\int_{t}^{r}e^{-\beta(r'-t)}\Big[q_{1}(X^{\Xi'}_{r'},r',\xi^{\Xi'}_{r'};\Xi)-\tilde{q}_{1}(X^{\Xi'}_{r'},r',\xi^{\Xi'}_{r'})\Big]dr'\\
&+\int_{t}^{r}e^{-\beta(r'-t)}\Big[q_{0}(X^{\Xi'}_{r'},r',\xi^{\Xi'}_{r'};\Xi)-\tilde{q}_{0}(X^{\Xi'}_{r'},r',\xi^{\Xi'}_{r'})\Big]d(\xi^{\Xi'})^{c}_{r'}\\
&+\sum\limits_{r'\in[t,r]}e^{-\beta(r'-t)}\int_{0}^{\Delta\xi^{\Xi'}_{r'}}\Big[q_{0}(X^{\Xi'}_{r'-}-u,r',\xi^{\Xi'}_{r'-}+u;\Xi)-\tilde{q}_{0}(X^{\Xi'}_{r'-}-u,r',\xi^{\Xi'}_{r'-}+u)\Big]du\\
&+\int_{t}^{r}e^{-\beta(r'-t)}J_{x}(X^{\Xi'}_{r'},r',\xi^{\Xi'}_{r'};\Xi)\sigma(X^{\Xi'}_{r'},r',\xi^{\Xi'}_{r'})dB_{r'},
\end{aligned}
\end{equation*}
The direction (a)$\Rightarrow$(b) readily follows. For the direction (b)$\Rightarrow$(a), we have from (b) and the last formula that for any $(x,t,y)\in\mathcal{Q}$ and $r\ge t$, almost surely,
\begin{equation*}
\begin{aligned}
0=&\int_{t}^{r}e^{-\beta(r'-t)}\Big[q_{1}(X^{\Xi'}_{r'},r',\xi^{\Xi'}_{r'};\Xi)-\tilde{q}_{1}(X^{\Xi'}_{r'},r',\xi^{\Xi'}_{r'})\Big]dr'\\
&+\int_{t}^{r}e^{-\beta(r'-t)}\Big[q_{0}(X^{\Xi'}_{r'},r',\xi^{\Xi'}_{r'};\Xi)-\tilde{q}_{0}(X^{\Xi'}_{r'},r',\xi^{\Xi'}_{r'})\Big]\circ d\xi^{\Xi'}_{r'}.
\end{aligned}
\end{equation*}
Setting $(x,t,y)\in (W^{\Xi'})^{c}$ and letting $r=t$ in the last formula, we deduce (\ref{a1}). Setting $(x,t,y)\in W^{\Xi'}$, then the second term on the right-hand side vanishes when $r$ is sufficiently close to $t$, and we deduce (\ref{a2}), which finishes the proof of the claim.

The statement $(1)$ of the theorem follows directly by letting $\Xi'=\Xi$ in the assertion. 

For the statement $(2)$,  (i)$\Rightarrow$(ii) follows directly by applying the It\^{o}'s formula, and (ii)$\Rightarrow$(iii) is trivial as the generated singular control $\xi^{\Xi'}$ is admissible. For (iii)$\Rightarrow$(i), using the assertion, we obtain that for any admissible $\Xi'$, (\ref{a1})$\sim$(\ref{a2}) hold. By varying $\Xi'$, we obtain that
\begin{equation*}
\begin{aligned}
&\tilde{q}_{0}(x,t,y)=q_{0}(x,t,y;\Xi),\quad \forall(x,t,y)\in \mathcal{Q},\\
&\tilde{q}_{1}(x,t,y)=q_{1}(x,t,y;\Xi),\quad \forall(x,t,y)\in \mathcal{Q}.
\end{aligned}
\end{equation*}
\quad
\end{proof}

\begin{theorem}[Characterization of $J,q_{0},q_{1}$ for a given singular control law]
\label{qforjq}
Given an admissible singular control law $\Xi$, a function $\tilde{J}\in\mathcal{R}^{\Xi}$, and continuous functions $\tilde{q}_{0},\tilde{q}_{1}:\mathcal{Q}\rightarrow \mathbb{R}$ with the property
\begin{equation*}
\begin{aligned}
&\tilde{J}(x,T,y)=\bar{F}(x,y;\Xi),\\
&\tilde{q}_{0}(x,t,y)=0,\ \forall (x,t,y)\in (W^{\Xi})^{c},\\
&\tilde{q}_{1}(x,t,y)=0,\ \forall (x,t,y)\in W^{\Xi}.
\end{aligned}
\end{equation*}
Then, the following three statements are equivalent:\\
(i) $\tilde{J}$, $\tilde{q}_{0}$ and $\tilde{q}_{1}$ are the value function and q-functions of $\Xi$.\\
(ii) For any admissible singular control $\xi'$, the process
\begin{equation*}
\begin{aligned}
&\bigg\{e^{-\beta(r-t)}\tilde{J}(X^{\xi'}_{r},r,\xi'_{r})+\int_{t}^{r}e^{-\beta(r'-t)}\Big[H(X^{\xi'}_{r'},r',\xi'_{r'})-\tilde{q}_{1}(X^{\xi'}_{r'},r',\xi'_{r'})\Big]dr'\\
&+\int_{t}^{r}e^{-\beta(r'-t)}\Big[c(X^{\xi'}_{r'},r',\xi'_{r'})-\tilde{q}_{0}(X^{\xi'}_{r'},r',\xi'_{r'})\Big]\circ d\xi'_{r'},\  r\in[t,T]\bigg\} 
\end{aligned}
\end{equation*}
is an $\{\mathcal{F}_{r}\}_{r\in[t,T]}$-martingale.\\
(iii) For any admissible singular control law $\Xi'$, the process
\begin{equation*}
\begin{aligned}
&\bigg\{e^{-\beta(r-t)}\tilde{J}(X^{\Xi'}_{r},r,\xi^{\Xi'}_{r})+\int_{t}^{r}e^{-\beta(r'-t)}\Big[H(X^{\Xi'}_{r'},r',\xi^{\Xi'}_{r'})-\tilde{q}_{1}(X^{\Xi'}_{r'},r',\xi^{\Xi'}_{r'})\Big]dr'\\
&+\int_{t}^{r}e^{-\beta(r'-t)}\Big[c(X^{\Xi'}_{r'},r',\xi^{\Xi'}_{r'})-\tilde{q}_{0}(X^{\Xi'}_{r'},r',\xi^{\Xi'}_{r'})\Big]\circ d\xi^{\Xi'}_{r'},\  r\in[t,T]\bigg\} 
\end{aligned}
\end{equation*}
is an $\{\mathcal{F}_{r}\}_{r\in[t,T]}$-martingale.
\end{theorem}
\begin{proof}
The direction (ii)$\Rightarrow$(iii) is trivial, we show (i)$\Rightarrow$(ii) and (iii)$\Rightarrow$(i).

For (i)$\Rightarrow$(ii), applying It\^{o}-Tanaka-Meyer's formula to $\big\{e^{-\beta(r'-t)}\tilde{J}(X^{\xi'}_{r'},r',\xi'_{r'})\big\}_{r'\in[t,r]}$ yields
\begin{align*}
&e^{-\beta(r-t)}\tilde{J}(X^{\xi'}_{r},r,\xi'_{r})-\tilde{J}(x,t,y)+\int_{t}^{r}e^{-\beta(r'-t)}\Big[H(X^{\xi'}_{r'},r',\xi'_{r'})-\tilde{q}_{1}(X^{\xi'}_{r'},r',\xi'_{r'})\Big]dr'\\&+\int_{t}^{r}e^{-\beta(r'-t)}\Big[c(X^{\xi'}_{r'},r',\xi'_{r'})-\tilde{q}_{0}(X^{\xi'}_{r'},r',\xi'_{r'})\Big]\circ d\xi'_{r'}\\
=&\int_{t}^{r}e^{-\beta(r'-t)}\Big[H(X^{\xi'}_{r'},r',\xi'_{r'})-\beta\tilde{J}(X^{\xi'}_{r'},r',\xi'_{r'})+\tilde{J}_{t}(X^{\xi'}_{r'},r',\xi'_{r'})+\mu(X^{\xi'}_{r'},r',\xi'_{r'})\tilde{J}_{x}(X^{\xi'}_{r'},r',\xi'_{r'})\\
&+\frac{1}{2}\sigma^{2}(X^{\xi'}_{r'},r',\xi'_{r'})\tilde{J}_{xx}(X^{\xi'}_{r'},r',\xi'_{r'})-\tilde{q}_{1}(X^{\xi'}_{r'},r',\xi'_{r'})\Big]dr'\\
&+\int_{t}^{r}e^{-\beta(r'-t)}\Big[c(X^{\xi'}_{r'},r',\xi'_{r'})-\tilde{J}_{x}(X^{\xi'}_{r'},r',\xi'_{r'})+\tilde{J}_{y}(X^{\xi'}_{r'},r',\xi'_{r'})-\tilde{q}_{0}(X^{\xi'}_{r'},r',\xi'_{r'})\Big]d(\xi')^{c}_{r'}\\
&+\!\!\!\sum\limits_{r'\in[t,r]}\!e^{-\beta(r'-t)}\!\int_{0}^{\Delta\xi'_{r'}}\!\!\!\Big[c(X^{\xi'}_{r'-\!}\!-\!u,r',\xi'_{r'-\!}\!+\!u)\!-\!\tilde{J}_{x}(X^{\xi'}_{r'-\!}\!-\!u,r',\xi'_{r'-\!}\!+\!u)\!+\!\tilde{J}_{y}(X^{\xi'}_{r'-\!}\!-\!u,r',\xi'_{r'-\!}\!+\!u)\\&\!-\!\tilde{q}_{0}(X^{\xi'}_{r'-\!}\!-\!u,r',\xi'_{r'-\!}\!+\!u)\Big]du+\int_{t}^{r}e^{-\beta(r'-t)}\tilde{J}_{x}(X^{\xi'}_{r'},r',\xi'_{r'})\sigma(X^{\xi'}_{r'},r',\xi'_{r'})dB_{r'}.
\end{align*}
Then (i)$\Rightarrow$(ii) follows. For (iii)$\Rightarrow$(i), similar to the proof of (b)$\Rightarrow$(a) in proof of Theorem \ref{qforq}, we obtain that
\begin{equation*}
\begin{aligned}
\tilde{q}_{1}(x,t,y)=&H(x,t,y)-\beta \tilde{J}(x,t,y)+\tilde{J}_{t}(x,t,y)+\mu(x,t,y)\tilde{J}_{x}(x,t,y)\\&+\frac{1}{2}\sigma^{2}(x,t,y)\tilde{J}_{xx}(x,t,y),\quad\forall (x,t,y)\in W^{\Xi'},\\
\tilde{q}_{0}(x,t,y)=&c(x,t,y)-\tilde{J}_{x}(x,t,y)+\tilde{J}_{y}(x,t,y),\quad\forall (x,t,y)\in (W^{\Xi'})^{c}.
\end{aligned}
\end{equation*}
By varying $\Xi'$, we have
\begin{equation*}
\begin{aligned}
\tilde{q}_{1}(x,t,y)=&H(x,t,y)-\beta \tilde{J}(x,t,y)+\tilde{J}_{t}(x,t,y)+\mu(x,t,y)\tilde{J}_{x}(x,t,y)\\
&+\frac{1}{2}\sigma^{2}(x,t,y)\tilde{J}_{xx}(x,t,y),\quad\forall (x,t,y)\in \mathcal{Q},\\
\tilde{q}_{0}(x,t,y)=&c(x,t,y)-\tilde{J}_{x}(x,t,y)+\tilde{J}_{y}(x,t,y),\quad\forall (x,t,y)\in \mathcal{Q}.
\end{aligned}
\end{equation*}
Then, $\tilde{J}$ solves the variational equality in Theorem \ref{verif-a}. Applying Theorem \ref{verif-a}, we obtain that $\tilde{J}$ is the value function under $\Xi$ and the conclusion follows.
\end{proof}

\begin{theorem}[Simultaneous characterization of the optimal value function and optimal q-functions]\label{qforopjq}Given a function $J^{*}\in\mathcal{R}^{*}$, and continuous functions $q^{*}_{0},q^{*}_{1}:\mathcal{Q}\rightarrow \mathbb{R}$ with the property that $\Xi^{*}$ defined by $W^{\Xi^{*}}:=\big\{(x,t,y)\big|q^{*}_{0}(x,t,y)>0\big\}$ is admissible and 
\begin{equation*}
\begin{aligned}
&J^{*}(x,T,y)=\tilde{F}(x,y),\\
&\min\big\{q^{*}_{1}(x,t,y),q^{*}_{0}(x,t,y)\big\}=0,
\end{aligned}
\end{equation*}
then the following three statements are equivalent:\\
(i) $J^{*}$, $q^{*}_{0}$,  $q^{*}_{1}$ are respectively the optimal value function and q-functions.\\
(ii) For any admissible singular control $\xi$,
\begin{equation}
\label{m3_a}
\begin{aligned}
&\bigg\{e^{-\beta(r-t)}J^{*}(X^{\xi}_{r},r,\xi_{r})+\int_{t}^{r}e^{-\beta(r'-t)}\Big[H(X^{\xi}_{r'},r',\xi_{r'})-q^{*}_{1}(X^{\xi}_{r'},r',\xi_{r'})\Big]dr'\\
&+\int_{t}^{r}e^{-\beta(r'-t)}\Big[c(X^{\xi}_{r'},r',\xi_{r'})-q^{*}_{0}(X^{\xi}_{r'},r',\xi_{r'})\Big]\circ d\xi_{r'},\  r\in[t,T]\bigg \}
\end{aligned}
\end{equation}
is an $\{\mathcal{F}_{r}\}_{r\in[t,T]}$-martingale.\\
(iii) For any admissible singular control law $\Xi$,
\begin{equation}
\label{m3_b}
\begin{aligned}
&\bigg\{e^{-\beta(r-t)}J^{*}(X^{\Xi}_{r},r,\xi^{\Xi}_{r})+\int_{t}^{r}e^{-\beta(r'-t)}\Big[H(X^{\Xi}_{r'},r',\xi^{\Xi}_{r'})-q^{*}_{1}(X^{\Xi}_{r'},r',\xi^{\Xi}_{r'})\Big]dr'\\
&+\int_{t}^{r}e^{-\beta(r'-t)}\Big[c(X^{\Xi}_{r'},r',\xi^{\Xi}_{r'})-q^{*}_{0}(X^{\Xi}_{r'},r',\xi^{\Xi}_{r'})\Big]\circ d\xi^{\Xi}_{r'},\  r\in[t,T]\bigg \}
\end{aligned}
\end{equation}
is an $\{\mathcal{F}_{r}\}_{r\in[t,T]}$-martingale.
\end{theorem}
\begin{proof}
We only need to show (i)$\Rightarrow$(ii) and (iii)$\Rightarrow$(i).

(i)$\Rightarrow$(ii) still follows from It\^{o}'s formula. Now we show (iii)$\Rightarrow$(i). Assuming that (\ref{m3_b}) is an $\{\mathcal{F}_{r}\}_{r\in[t,T]}$-martingale for any admissible $\Xi$, similar to the proof of (2) of Theorem \ref{qforq}, one can obtain that
\begin{equation*}
\begin{aligned}
q^{*}_{0}(x,t,y)=&c(x,t,y)-J^{*}_{x}(x,t,y)+J^{*}_{y}(x,t,y),\quad\forall (x,t,y)\in\mathcal{Q},\\
q^{*}_{1}(x,t,y)=&-\beta J^{*}(x,t,y)+J^{*}_{t}(x,t,y)+\mu(x,t,y)J^{*}_{x}(x,t,y)\\
&+\frac{1}{2}\sigma^{2}(x,t,y)J^{*}_{xx}(x,t,y)+H(x,t,y),\quad\forall (x,t,y)\in\mathcal{Q}.
\end{aligned}
\end{equation*}
Now $J^{*}$ solves 
\begin{equation*}
\begin{aligned}
&J^{*}(x,T,y)=\tilde{F}(x,y),\\
&\begin{aligned}
\min\big\{&-\beta J^{*}(x,t,y)+J^{*}_{t}(x,t,y)+\mu(x,t,y)J^{*}_{x}(x,t,y)\\
&+\frac{1}{2}\sigma^{2}(x,t,y)J^{*}_{xx}(x,t,y)+H(x,t,y),c(x,t,y)-J^{*}_{x}(x,t,y)+J^{*}_{y}(x,t,y)\big\}=0.
\end{aligned}
\end{aligned}
\end{equation*}
Applying Theorem \ref{verif} yields that $J^{*}$ is the optimal value function.
\end{proof}

We can use the PE method to update the q-functions in Theorem \ref{qforq} or update the optimal value function and the q-functions simultaneously in Theorem \ref{qforopjq}.

\subsection{Extension to infinite-horizon problem}
\label{infitime}
The above RL theory for finite-time singular control problems can be extended to infinite-time singular control problems with some minor modifications. In the infinite-time setting, the problem is defined by
\begin{equation}
\label{object1}
\begin{aligned}
&\min_{\xi }J(x,t,y;\xi):=\mathbb{E}_{x,t,y}\bigg[\int_{t}^{\infty}e^{-\beta(r-t)}H(X^{\xi}_{r},r,\xi_{r})dr\\&+\int_{t}^{\infty}e^{-\beta(r-t)}c(X^{\xi}_{r},r,\xi_{r})\circ d\xi_{r}\bigg], \quad 
\forall(x,t,y)\in\mathcal{Q}:=\mathbb{R}\times[0,+\infty)\times[0,+\infty).
\end{aligned}
\end{equation}
The definitions in Subsection 2.1 can be adapted to infinite-time setting by replacing $T$ with $\infty$.
The verification theorem is slightly different as we replace the integrability condition (3) in Theorem \ref{verif} with the transversality condition (3) in Theorem \ref{verif1} for the infinite-time setting.
\begin{theorem}[verification theorem]
\label{verif1}
Given a function $V(x,t,y)$ which has derivatives $V_{x}$ and $V_{y}$ on $\mathcal{Q}$, define $\hat{\Xi}=(W^{\hat{\Xi}},(W^{\hat{\Xi}})^{c})$ by
\begin{equation*}
W^{\hat{\Xi}}:=\big\{(x,t,y)\in\mathcal{Q}\big|c(x,t,y)-V_{x}(x,t,y)+V_{y}(x,t,y)>0\big\}.
\end{equation*}
Assume that the following conditions hold:
(1) $V\in C^{2,1,1}(\mathcal{Q})$ and the following HJB equations hold:
\begin{equation*}
\begin{aligned}
&\!\min\!\big\{\!H(x,t,y)\!-\!\beta V(x,t,y)\!+\!V_{t}(x,t,y)\!+\!\mu(x,t,y)V_{x}(x,t,y)\!+\!\frac{1}{2}\sigma^{2}(x,t,y)V_{xx}(x,t,y),\!\\
&\qquad\ c(x,t,y)-V_{x}(x,t,y)+V_{y}(x,t,y)\big\}=0,\quad \forall (x,t,y)\in\mathcal{Q}.
\end{aligned}
\end{equation*}
(2) $\hat{\Xi}$ is an admissible singular control law.\\
(3) For any admissible $\xi$,
\begin{equation}
\label{trans}
\lim\limits_{r\rightarrow\infty}\mathbb{E}_{x,t,y}e^{-\beta(r-t)}V(X^{\xi}_{r},r,\xi_{r})=0.
\end{equation}
Then the singular control generated by $\hat{\Xi}$ at $(x,t,y)$, denoted by $\xi^{x,t,y,\hat{\Xi}}$, is an optimal singular control for Problem (\ref{object1}) starting at $(x,t,y)$. Moreover, $V(x,t,y)$ is the value function of $\xi^{x,t,y,\hat{\Xi}}$ and thus the optimal value function for Problem (\ref{object1})  starting at $(x,t,y)$.
\end{theorem}

The previous PE, PI, and q-learning theorems are directly applicable to the infinite-horizon problem by replacing the terminal conditions with transversality conditions similar to (\ref{trans}).

\section{q-Learning Algorithms and Numerical Examples}
\label{algorithm}
This section applies the theoretical results in the previous section to devise some continuous-time model-free RL algorithms for singular controls.

Subsection \ref{4gen} introduces three learning schemes in the general case, covering both finite and infinite horizon problems. Then, Subsection \ref{4case} illustrates the numerical performance of a simplified algorithm in a concrete example.

\subsection{General case: q-learning algorithms}
\label{4gen}

In the general case, the previous martingale characterization results motivate three learning schemes to learn the optimal singular control law.

The first two schemes follow the actor-critic approach. In each iteration, the agent evaluates all the functions needed, including the value function and the q functions, and then updates the policy. The first learning scheme is to evaluate the value function and the q functions successively using Theorem \ref{PET} and Theorem \ref{qforq}. The second learning scheme improves the first one by evaluating the value function and the q functions simultaneously, using Theorem \ref{qforjq}.

The last scheme differs from the above two in that it is critic-only. Benefited from Theorem \ref{qforopjq}, the third learning scheme directly learns the optimal value function, omitting the step to update the policy.

The above three learning schemes are summarized as follows:
\begin{itemize}
\item {\bf Scheme-1:} In each iteration, first evaluate the value function by PE Theorem \ref{PET}, then apply q-learning Theorem \ref{qforq} to evaluate the q functions,  finally update the policy by PI Theorem \ref{PIT}.
\item{\bf Scheme-2:} In each iteration, first simultaneously update the value function and q functions by Theorem \ref{qforjq}, then update the policy by PI Theorem \ref{PIT}.
\item{\bf Scheme-3:} In each iteration, update the optimal value function and q functions by Theorem \ref{qforopjq}. 
\end{itemize}

\begin{remark}
We emphasize some differences between the three schemes. First, {\bf Scheme-1} has the most steps in each iteration, but almost no constraints on the structure of the value function and q functions, while {\bf Scheme-3} only has one step in each iteration, but also the most constraints on the function structure, i.e., it requires that
\begin{equation*}
\begin{aligned}
&J^{*}(x,T,y)=\tilde{F}(x,y),\\
&\min\big\{q^{*}_{1}(x,t,y),q^{*}_{0}(x,t,y)\big\}=0.
\end{aligned}
\end{equation*}
Second, the value function and q functions in {\bf Scheme-3} differ from the other two in that they do not rely on the policy, which presents the optimal case.
\end{remark}

\begin{remark}
The policy improvement or even convergence in {\bf Scheme-1} and {\bf Scheme-2} relies on the admissibility of updated policies during the whole iteration process as required by Theorem \ref{PIT}. In specific problems where only environment parameters are unknown, such admissibility can possibly be guaranteed. Sometimes, even convergence can be proved; see our example in Subsection \ref{4case}. However, in model-free problems with no known model structure, the value functions and q functions are usually parameterized by neural networks, and theoretical convergence is usually hard to obtain.
\end{remark}

To facilitate the above learning schemes, both PE and PI need to be implemented. The PI can be directly implemented once the q functions are correctly evaluated. For the PE of the value function and q functions, we can apply the following two types of loss functions introduced in \cite{jia2022policy} to utilize the martingale characterizations in Theorems \ref{PET}, \ref{qforq}, \ref{qforjq}, and \ref{qforopjq}:
\begin{itemize}
\item[\textbf{(a)}] Minimizing the square-error martingale loss function.
\item[\textbf{(b)}] Applying the martingale orthogonality condition.
\end{itemize}

In singular control problems, the calculation of the martingale loss in (a) and the integration w.r.t the martingale in (b) differ from the counterparts in regular control problems. The continuous part and the jumps in the martingale process have to be treated separately, as shown in the following two example algorithms.
\vskip 5pt
As the first example, for the finite horizon problem on $[0, T]$, we present an offline martingale-loss q-learning algorithm by combining the {\bf Scheme-1} with approach (a) for utilizing martingale characterization. Suppose that we have $N$ mesh grids $0=t_{0}<t_{1}<\cdots \ t_{N}=T$ discretizing the time horizon $[0, T]$, and denote the running reward/penalty and control cost in each step by
\begin{equation*}
H_{t_{n}}: =\int_{t_{n}}^{t_{n+1}}e^{-\beta r}H(X_{r}, r, \xi_{r})dr, \quad c_{t_{n}}: =\int_{t_{n}}^{t_{n+1}-}e^{-\beta r}c(X_{r}, r, \xi_{r})\circ d\xi_{r}, \quad \forall n<N,
\end{equation*}
denote the terminal reward/penalty at $T-$ by $F_{t_{N}-}=F_{T-}: =F(X_{T-}, \xi_{T-})$
All $H_{t_{n}}$, $c_{t_{n}}$ and $F_{t_{N}-}$ are observable in practice. Suppose the functions $J,q_{1},q_{0}$ are parameterized by $\theta^{J},\theta^{q_{0}},\theta^{q_{1}}$. For the martingale $\{M_{r-}\}_{r\in[0,T]}$ in Theorem \ref{PET}, we have
\begin{equation*}
M_{t_{N}-}-M_{t_{n}-}= e^{-\beta T}F_{t_{N}-}-e^{-\beta t_{n}}J^{\theta^{J}}(X_{t_{n}-}, t_{n}, \xi_{t_{n}-};\Xi)+\sum\limits_{i=n}^{N-1}H_{t_{i}}+\sum\limits_{i=n}^{N-1}c_{t_{i}}.
\end{equation*}
According to \cite{jia2022policy}, for the loss function in (a), the updated parameters $\theta^{J}$ should minimize
 \begin{equation*}
 \begin{aligned}
 \frac{1}{2}\mathbb{E}\int_{t}^{T}[M_{T-}-M_{r-}]^{2}dr\approx&\frac{1}{2}\mathbb{E}\sum\limits_{n=0}^{N-1}[M_{t_{N}-}-M_{t_{n}-}]^{2}\Delta t\\
 =&\frac{1}{2}\mathbb{E}\sum\limits_{n=0}^{N-1}[e^{-\beta T}F_{t_{N}-}-e^{-\beta t_{n}}J^{\theta^{J}}(X_{t_{n}-}, t_{n}, \xi_{t_{n}-};\Xi)+\sum\limits_{i=n}^{N-1}H_{t_{i}}+\sum\limits_{i=n}^{N-1}c_{t_{i}}]^{2}\Delta t
 \end{aligned}
 \end{equation*}
 over $\theta^{J}$. The standard approach is to perform gradient descent with a learning rate to control the pace of the update. Suppose that the initial learning rate is $\alpha^{J}$ and the learning rate at episode $m$ is $\alpha^{J}l(m)$, with $l(\cdot)$ satisfying $l(1)=1$ being called the learning rate schedule function in \cite{jia2023q}. Then, the parameter $\theta^{J}$ at episode $m$ should be updated by
\begin{equation}
\label{iJ}
\begin{aligned}
\theta^{J}\leftarrow\theta^{J}+\alpha^{J}l(m)\Delta t\sum\limits_{n=0}^{N-1}\Big[&e^{-\beta T}F_{t_{N}-}-e^{-\beta t_{n}}J^{\theta^{J}}(X_{t_{n}-}, t_{n}, \xi_{t_{n}-};\Xi)+\sum\limits_{i=n}^{N-1}c_{t_{i}}\\&+\sum\limits_{i=n}^{N-1}H_{t_{i}}\Big]e^{-\beta t_{n}}\frac{\partial}{\partial\theta^{J}}J^{\theta^{J}}(X_{t_{n}-}, t_{n}, \xi_{t_{n
}-};\Xi).
\end{aligned}
\end{equation}

Similarly, we can compute the martingale difference for the martingale in Theorem \ref{qforq}. The iteration steps for the parameters $\theta^{q_{0}}$ and $\theta^{q_{1}}$ are then obtained as:
\begin{equation}
\label{iq1}
\begin{aligned}
\theta^{q_{1}}\leftarrow&\theta^{q_{1}}+\alpha^{q_{1}}l(m)(\Delta t)^{2}\sum\limits_{n=0}^{N-1}\bigg\{\Big[e^{-\beta T}F_{t_{N}-}-e^{-\beta t_{n}}J^{\theta^{J}}(X_{t_{n}-}, t_{n}, \xi_{t_{n}-};\Xi)+\sum\limits_{i=n}^{N-1}\big[c_{t_{i}}\\&
-e^{-\beta t_{i}}q_{0}^{\theta^{q_{0}}}(X_{t_{i}-}, t_{i}, \xi_{t_{i}-};\Xi)\Delta\xi_{t_{i}}
-e^{-\beta t_{i}}q_{0}^{\theta^{q_{0}}}(X_{t_{i}}, t_{i}, \xi_{t_{i}};\Xi)(\xi_{t_{i+1}-}-\xi_{t_{i}})\big]\\
&+\sum\limits_{i=n}^{N-1}\big[H_{t_{i}}-e^{-\beta t_{i}}q_{1}^{\theta^{q_{1}}}(X_{t_{i}}, t_{i}, \xi_{t_{i}};\Xi)\Delta t\big]\Big]\sum\limits_{i=n}^{N-1}e^{-\beta t_{i}}\frac{\partial}{\partial\theta^{q_{1}}}q_{1}^{\theta^{q_{1}}}(X_{t_{i}}, t_{i}, \xi_{t_{i}};\Xi)\bigg\}, 
\end{aligned}
\end{equation}
\begin{equation}
\label{iq0}
\begin{aligned}
&\theta^{q_{0}}\leftarrow\theta^{q_{0}}+\alpha^{q_{0}}l(m)\Delta t\sum\limits_{n=0}^{N-1}\bigg\{\Big[e^{-\beta T}F_{t_{N}-}-e^{-\beta t_{n}}J^{\theta^{J}}(X_{t_{n}-}, t_{n}, \xi_{t_{n}-};\Xi)+\sum\limits_{i=n}^{N-1}\big[c_{t_{i}}\\&
-e^{-\beta t_{i}}q_{0}^{\theta^{q_{0}}}(X_{t_{i}-}, t_{i}, \xi_{t_{i}-};\Xi)\Delta\xi_{t_{i}}-e^{-\beta t_{i}}q_{0}^{\theta^{q_{0}}}(X_{t_{i}}, t_{i}, \xi_{t_{i}};\Xi)(\xi_{t_{i+1}-}-\xi_{t_{i}})\big]\\&+\sum\limits_{i=n}^{N-1}\big[H_{t_{i}}-e^{-\beta t_{i}}q_{1}^{\theta^{q_{1}}}(X_{t_{i}}, t_{i}, \xi_{t_{i}};\Xi)\Delta t\big]\Big]
\sum\limits_{i=n}^{N-1}e^{-\beta t_{i}}\big[\frac{\partial}{\partial\theta^{q_{0}}}q_{0}^{\theta^{q_{0}}}(X_{t_{i}-}, t_{i}, \xi_{t_{i}-};\Xi)\Delta\xi_{t_{i}}\\&\qquad\qquad\qquad\qquad\qquad\qquad+\frac{\partial}{\partial\theta^{q_{0}}}q_{0}^{\theta^{q_{0}}}(X_{t_{i}}, t_{i}, \xi_{t_{i}};\Xi)(\xi_{t_{i+1}-}-\xi_{t_{i}})\big]\bigg\}. 
\end{aligned}
\end{equation}
Combining the above PE implementation with the PI in Theorem \ref{PIT} yields the following Algorithm \ref{ml-example}. The algorithm is offline, because in each episode, the parameters and the policy are updated using the data of the whole episode. 
\begin{algorithm}[h]
\caption{Offline q-learning algorithm using martingale loss (finite-horizon)}
\label{ml-example}
\renewcommand{\algorithmicrequire}{\textbf{INPUT:}}
\begin{algorithmic}
\REQUIRE
Initial state $x_{0}$ and singular control value $y_{0}$, terminal time $T>0$, number of episodes $M$, number of mesh grids $N$, proper parameterization $\phi^{\theta_{\phi}}(\cdot,\cdot,\cdot;\cdot)$ for $\phi=J,q_{0},q_{1}$, initial learning rates $\alpha^{J},\alpha^{q_{0}},\alpha^{q_{1}}$ and a learning rate schedule function $l(\cdot)$. 
\STATE Initialize $\theta^{\phi}$ for $\phi=J,q_{0},q_{1}$, and initialize $\Xi$. Calculate the time step $\Delta t=\frac{T}{N}$ and set $t_{n}: =n\frac{T}{N}, n=0, 1, 2, \cdots, N$.
\FOR {episode $m=1,2,\cdots,M$}
\STATE Initialize $n=0$ and $(X_{t_{n}-}, t_{n}, \xi_{t_{n}-})=(x_{0}, 0, y_{0})$.
\STATE Simulate a training sample path, which can be collected either on-policy or off-policy. Specifically, for $n<N$, obtain the state-time-control triples $(X_{t_{n}}, t_{n}, \xi_{t_{n}})$ and $(X_{t_{n+1}-}, t_{n+1}, \xi_{t_{n+1}-})$, the reward/penalty $H_{t_{n}}$, and the control cost $c_{t_{n}}$. For $n=N$, obtain the terminal reward/penalty $F_{t_{n}-}$.
\STATE Update $\theta^{J},\theta^{q_{0}},\theta^{q_{1}}$ respectively by (\ref{iJ})$\sim$(\ref{iq0}).

\STATE Update $\Xi$ by $W^{\Xi}\leftarrow int(\{(x',t',y')|q_{1}^{\theta^{q_{1}}}(x',t',y';\Xi)\le 0\  \text{and}\ q_{0}^{\theta^{q_{0}}}(x',t',y';\Xi)\ge 0\})$.
\ENDFOR
\end{algorithmic}
\end{algorithm}

The second algorithm is designed for the infinite-horizon problem. We present an online temporal-difference q-learning algorithm combining the {\bf Scheme 3} and the martingale orthogonality condition in (b). Under {\bf Scheme 3}, only PE is needed for evaluating the optimal value function and corresponding q functions with Theorem \ref{qforopjq}. When using approach (b) for utilizing the martingale property, one can set the test function as the derivatives with respect to the parameter and use stochastic approximation to iterate. \cite{jia2022policy} shows that such setting leads to the TD(0) algorithm in \cite{sutton1988learning}. Taking the same choice for the test function and using stochastic approximation, we derive that in the online training the parameter $\theta^{\phi},\phi=J,q_{0},q_{1}$ at step $n$ of episode $m$ can be updated by
\begin{equation*}
\begin{aligned}
\theta^{\phi}\leftarrow&\theta^{\phi}+\alpha^{\phi}l(m)\int_{t_{n}}^{t_{n+1}-}\frac{\partial}{\partial \theta^{\phi}}\phi^{\theta^{\phi}}(X_{t},t,\xi_{t})dM^{\theta}_{t}\\
\approx&\theta^{\phi}+\alpha^{\phi}l(m)\Big[G^{c}_{n}\frac{\partial}{\partial \theta^{\phi}}\phi^{\theta^{\phi}}(X_{t_{n}},t_{n},\xi_{t_{n}})+G^{d}_{n}\frac{\partial}{\partial \theta^{\phi}}\phi^{\theta^{\phi}}(X_{t_{n}-},t_{n},\xi_{t_{n}-})\Big] ,\quad\phi=J,q_{0},q_{1},
\end{aligned}
\end{equation*}
where $M^{\theta}:=M^{(\theta^{J},\theta^{q_{0}},\theta^{q_{1}})}_{t}$ is the martingale in Theorem \ref{qforopjq} and $G^{c}_{n}$, $G^{d}_{n}$ are the continuous and discrete part of its  temporal difference at step $n$, i.e.,
\begin{align*}
M^{\theta}_{t}=&e^{-\beta t}J^{\theta^{J}}(X_{t},t,\xi_{t})+\int_{0}^{t}e^{-\beta r}\big[H(X_{r},r,\xi_{r})-q^{\theta^{q_{1}}}_{1}(X_{r},r,\xi_{r})\big]dr\\&+\int_{0}^{t}e^{-\beta r}\big[c(X_{r},r,\xi_{r})-q^{\theta^{q_{0}}}_{0}(X_{r},r,\xi_{r})\big]\circ d\xi_{r},\\
G^{c}_{n}: =&M^{\theta}_{t_{n+1}-}-M^{\theta}_{t_{n}}\\\approx&e^{-\beta t_{n+1}}J^{\theta^{J}}(X_{t_{n+1}-}, t_{n+1}, \xi_{t_{n+1}-})\!-\!e^{-\beta t_{n}}J^{\theta^{J}}(X_{t_{n}}, t_{n}, \xi_{t_{n}})\!+\!\big[H_{t_{n}}\!-\!e^{-\beta t_{n}}q_{1}^{\theta^{q_{1}}}(X_{t_{n}}, t_{n}, \xi_{t_{n}})\Delta t\big]\\&+\big[c^{c}_{t_{n}}-e^{-\beta t_{n}}q_{0}^{\theta^{q_{0}}}(X_{t_{n}}, t_{n}, \xi_{t_{n}})(\xi_{t_{n+1}-}-\xi_{t_{n}})\big] ,\\
G^{d}_{n}: =&M^{\theta}_{t_{n}}-M^{\theta}_{t_{n}-}\\=&e^{-\beta t_{n}}\big[J^{\theta^{J}}(X_{t_{n}}, t_{n}, \xi_{t_{n}})-J^{\theta^{J}}(X_{t_{n}-}, t_{n}, \xi_{t_{n}-})\big]+\big[c^{d}_{t_{n}}-e^{-\beta t_{n}}q_{0}^{\theta^{q_{0}}}(X_{t_{n}-}, t_{n}, \xi_{t_{n}-})\Delta\xi_{t_{n}}\big], 
\end{align*}
where $c^{c}_{t_{n}}$ and $c^{d}_{t_{n}}$ are respectively the continuous and discrete parts of the control cost in $[t_{n},t_{n+1})$, i.e., 
\begin{equation*}
\begin{aligned}
&c^{d}_{t_{n}}: =e^{-\beta t_{n}}c(X_{t_{n}-}, t_{n}, \xi_{t_{n}-})\Delta\xi_{t_{n}}, \quad \forall n<N, \\
&c^{c}_{t_{n}}: =\int_{t_{n}+}^{t_{n+1}-}e^{-\beta r}c(X_{r-}, r, \xi_{r-})d\xi_{r}=c_{t_{n}}-c^{d}_{t_{n}}, \quad \forall n<N. 
\end{aligned}
\end{equation*}
By convention, for learning infinite-horizon problems, we set a sufficiently large truncation time $T$ and train over the horizon $[0, T]$. The resulting algorithm is shown in Algorithm \ref{td0-example}. The algorithm is online as the agent can collect data and update the parameters immediately in each step during a single episode.

\begin{algorithm}[h]
\caption{ Online q-learning algorithm using temporal difference (infinite-horizon)}
\label{td0-example}
\renewcommand{\algorithmicrequire}{\textbf{INPUT:}}
\begin{algorithmic}
\REQUIRE
Initial state $x_{0}$ and singular control value $y_{0}$, truncation time $T>0$, number of episodes $M$, number of mesh grids $N$, proper parameterization $\phi^{\theta_{\phi}}(\cdot,\cdot,\cdot)$ for $\phi=J,q_{0},q_{1}$ in the optimal case, initial learning rates $\alpha^{J},\alpha^{q_{0}},\alpha^{q_{1}}$ and a learning rate schedule function $l(\cdot)$. 
\STATE Initialize $\theta^{\phi}$ for $\phi=J,q_{0},q_{1}$. Calculate the corresponding $\Xi^{*}$ and the time step $\Delta t=\frac{T}{N}$ and set $t_{n}: =n\frac{T}{N}, n=0, 1, 2, \cdots, N$.
\FOR {episode $m=1,2, \cdots, M$}
\STATE Initialize $n=0$ and $(X_{t_{n}-}, t_{n}, \xi_{t_{n}-})=(x_{0}, 0, y_{0})$.
\WHILE {$n<N$}
\STATE Apply $\Xi^{*}$ to the environment and obtain the state-time-control triples $(X_{t_{n}}, t_{n}, \xi_{t_{n}})$ and $(X_{t_{n+1}-}, t_{n+1}, \xi_{t_{n+1}-})$, the reward/penalty $H_{t_{n}}$, and the continuous and discrete control costs $c^{c}_{t_{n}}, c^{d}_{t_{n}}$.

\STATE Update $\theta^{J},\theta^{q_{0}},\theta^{q_{1}}$ by
\begin{equation*}
\begin{aligned}
\theta^{J}\leftarrow&\theta^{J}+\alpha^{J}l(m)\Big[G^{c}_{n}\frac{\partial}{\partial\theta^{J}}J^{\theta^{J}}(X_{t_{n}}, t_{n}, \xi_{t_{n}})+G^{d}_{n}\frac{\partial}{\partial\theta^{J}}J^{\theta^{J}}(X_{t_{n}-}, t_{n}, \xi_{t_{n}-})\Big], \\
\theta^{q_{1}}\leftarrow&\theta^{q_{1}}+\alpha^{q_{1}}l(m)\Big[G^{c}_{n}\frac{\partial}{\partial\theta^{q_{1}}}q_{1}^{\theta^{q_{1}}}(X_{t_{n}}, t_{n}, \xi_{t_{n}})+G^{d}_{n}\frac{\partial}{\partial\theta^{q_{1}}}q_{1}^{\theta^{q_{1}}}(X_{t_{n}-}, t_{n}, \xi_{t_{n}-})\Big], \\
\theta^{q_{0}}\leftarrow&\theta^{q_{0}}+\alpha^{q_{0}}l(m)\Big[G^{c}_{n}\frac{\partial}{\partial\theta^{q_{0}}}q_{0}^{\theta^{q_{0}}}(X_{t_{n}}, t_{n}, \xi_{t_{n}})+G^{d}_{n}\frac{\partial}{\partial\theta^{q_{0}}}q_{0}^{\theta^{q_{0}}}(X_{t_{n}-}, t_{n}, \xi_{t_{n}-})\Big].
\end{aligned}
\end{equation*}
\ENDWHILE
\ENDFOR
\end{algorithmic}
\end{algorithm}

\subsection{An infinite-horizon example with known control cost}
\label{4case}
We present a simple infinite horizon example where the control cost is known and the closed-form solution is available.

The example here is a simplification of the irreversible reinsurance problem in \cite{yan2022irreversible}, which studies a much more complex problem. We focus on such a simple example to better illustrate the implementation and effectiveness of our RL framework.

Suppose that $\mu,\sigma,c$ are constant and $H(x,t,y)=e^{ax}$ with constant $a>0,\beta>\mu a+\frac{1}{2}\sigma^{2}a^{2}$. The latter constraint is needed to ensure the finite value of the uncontrolled objective value. Assume that the model parameters $a,\mu$ and $\sigma$ are unknown. The HJB variational inequality can be written as
\begin{equation*}
\min\big\{H(x)-\beta V(x)+\mu V_{x}(x)+\frac{1}{2}\sigma^{2}V_{xx}(x),\ c-V_{x}(x)\big\}=0.
\end{equation*}
The general solution in the waiting region has the following form:
\begin{equation*}
V(x)=\frac{1}{\beta-\mu a-\frac{1}{2}\sigma^{2}a^{2}}e^{ax}+C_{1}e^{\lambda_{1}x}+C_{2}e^{\lambda_{2}x},
\end{equation*}
where
\begin{equation*}
\lambda_{1}=\frac{-\mu-\sqrt{\mu^{2}+2\beta\sigma^{2}}}{\sigma^{2}},\ \lambda_{2}=\frac{-\mu+\sqrt{\mu^{2}+2\beta\sigma^{2}}}{\sigma^{2}}.
\end{equation*}
By the boundary condition $\lim\limits_{x\rightarrow-\infty}V(x)=0$ and the fact that $\lambda_{1}<0<a<\lambda_{2}$, we have  $C_{1}=0$ and
\begin{equation*}
V(x)=\frac{1}{\beta-\mu a-\frac{1}{2}\sigma^{2}a^{2}}e^{ax}+C_{2}e^{\lambda_{2}x}
\end{equation*}
with $C_{2}$ to be determined. Conjecture that the waiting region is $\{(x,t,y)|x<\hat{x}\}$, then, using $V_{x}(\hat{x})=c$, $ V_{xx}(\hat{x})=0$, we obtain
{\setlength{\abovedisplayskip}{2pt}%
 \setlength{\belowdisplayskip}{2pt}%
\begin{align}
&\hat{x}=\frac{1}{a}\ln(\frac{\lambda_{2}c(\beta-\mu a-\frac{1}{2}\sigma^{2}a^{2})}{\lambda_{2}a-a^{2}}),\quad
 C_{2}=-\frac{a^{2}}{\lambda_{2}^{2}(\beta-\mu a-\frac{1}{2}\sigma^{2}a^{2})}e^{(a-\lambda_{2})\hat{x}}.\label{hx}
 \end{align}
}
One can verify that 
\begin{equation*}
V(x)=
\begin{cases}
\frac{1}{\beta-\mu a-\frac{1}{2}\sigma^{2}a^{2}}e^{ax}+C_{2}e^{\lambda_{2}x}, & x<\hat{x},\\
c(x-\hat{x})+V(\hat{x}), & x\ge \hat{x}
\end{cases}
\end{equation*}
satisfies the conditions in the verification theorem.

In this example, we can first rigorously show that the policy improvement iteration of the singular control laws indeed converges to the optimal singular control law starting from any initial singular control law.
\begin{proposition}
\label{converg}
In this example, given any initial singular control law $\Xi_{0}:=(W^{\Xi_{0}},(W^{\Xi_{0}})^{c})$, where $W^{\Xi_{0}}=\{x|x<x_{0}\}$ for some $x_{0}$, let $\{\Xi_{n}:=(W^{\Xi_{n}},(W^{\Xi_{n}})^{c})\}$ be the sequence of singular control laws generated by the PI Theorem \ref{PIT}, there exist a sequence $\{x_{n}\}$ such that $W^{\Xi_{n}}=\{x|x<x_{n}\}$ and $\lim\limits_{n\rightarrow\infty}x_{n}=\hat{x}$, where $\hat{x}$ is given by (\ref{hx}). In other words, the sequence $\{\Xi_{n}\}$ converges to the optimal singular control law $\hat{\Xi}:=(W^{\hat{\Xi}},(W^{\hat{\Xi}})^{c})$ with $W^{\hat{\Xi}}=\{x|x<\hat{x}\}$. 
\end{proposition}
\begin{proof}
Given any $\Xi=(W^{\Xi},(W^{\Xi})^{c})$ with  $W^{\Xi}=\{x|x<\bar{x}\}$, according to Theorem \ref{verif-a}, the value function $J=J(x;\Xi)$ is characterized by 
\begin{equation*}
\begin{cases}
e^{ax}-\beta J+\mu J_{x}+\frac{1}{2}\sigma^{2}J_{xx}=0, & x<\bar{x},\\
c-J_{x}=0, &  x\ge\bar{x}. 
\end{cases}
\end{equation*}
Using the $C^{1}$ smooth-fit condition, we have 
$J(x;\Xi)=K_{1}e^{ax}+K_{2}(\bar{x})e^{\lambda_{2} x}$ with 
 \begin{equation*}
K_{1}=\frac{1}{\beta-\mu a-\frac{1}{2}\sigma^{2}a^{2}},\ K_{2}(\bar{x})=\frac{c-K_{1}ae^{a\bar{x}}}{\lambda_{2}e^{\lambda_{2}\bar{x}}}.
\end{equation*}
We have the following results:\\
a. If $K_{2}(\bar{x})\ge 0$, then $\bar{x}<\hat{x}$. This is because $K_{2}(\hat{x})<0\le K_{2}(\bar{x})$ leads to $K_{1}e^{a\hat{x}}>K_{1}e^{a\bar{x}}$.\\
b. If $K_{2}(\bar{x})<0$ then clearly $\bar{x}$ is a solution of $K_{1}ae^{ax}+K_{2}(\bar{x})\lambda_{2}e^{\lambda_{2}x}=c$. Moreover, if $K_{1}ae^{ax}+K_{2}(\bar{x})\lambda_{2}e^{\lambda_{2}x}=c$ have two different solutions $x^{(1)}<x^{(2)}$, then $\bar{x}$ equals either $x^{(1)}$ or $x^{(2)}$, $K_{2}(x^{(1)})=K_{2}(x^{(2)})$, and $x^{(1)}<\hat{x}<x^{(2)}$; if $K_{1}ae^{ax}+K_{2}(\bar{x})\lambda_{2}e^{\lambda_{2}x}=c$ has only one solution $\bar{x}$, then $\bar{x}=\hat{x}$. This is because $K_{2}(\cdot)$ is first decreasing then increasing and $\hat{x}$ is the minimum of $K_{2}(\cdot)$.

Recall that the q-functions here are
\begin{equation*}
\begin{aligned}
&q_{0}(x;\Xi)=c-J_{x}(x;\Xi),\\
&q_{1}(x;\Xi)=e^{ax}-\beta J(x;\Xi)+\mu J_{x}(x;\Xi)+\frac{1}{2}\sigma^{2}J_{xx}(x;\Xi).
\end{aligned}
\end{equation*}
Now we have three cases to discuss:\\
(1) Case $\bar{x}>\hat{x}$: In this case, $K_{2}(\bar{x})<0$ and $x^{(1)}<x^{(2)}=\bar{x}$. Then $q_{0}(x;\Xi)\ge 0$ holds for $x\le x^{(1)}$ and $x\ge x^{(2)}$. And we have $\lim\limits_{x\uparrow x^{(2)}}J_{xx}(x,\Xi)<0$, which leads to $0=\lim\limits_{x\uparrow x^{(2)}}q_{1}(x;\Xi)<\lim\limits_{x\downarrow x^{(2)}}q_{1}(x;\Xi)$. Thus $q_{1}(x;\Xi)\ge \lim\limits_{x\downarrow x^{(2)}}q_{1}(x;\Xi)>0,\forall x\ge x^{(2)}$. Therefore, the improved waiting region $\tilde{W}$ in Theorem \ref{PIT} is given by
\begin{equation*}
\tilde{W}=int\Big(\big\{x|q_{0}(x;\Xi)\ge 0,q_{1}(x;\Xi)\le 0\big\}\Big)=\{x|x<x^{(1)}\}.
\end{equation*}
(2) Case $\bar{x}<\hat{x}$ with $K_{2}(\bar{x})<0$. In this case, we have $\bar{x}=x^{(1)}<x^{(2)}$ and $q_{0}(x;\Xi)\ge 0$ holds for any $x\in\mathbb{R}$. Moreover, 
    $\lim\limits_{x\uparrow x^{(1)}}J_{xx}(x;\Xi)>0$, which yields $0=\lim\limits_{x\uparrow x^{(1)}}q_{1}(x;\Xi)>\lim\limits_{x\downarrow x^{(1)}}q_{1}(x;\Xi)$. Besides, define $\Xi^{x^{(2)}}:=(\{x|x\le x^{(2)}\},\{x|x> x^{(2)}\})$, then $\lim\limits_{x\uparrow x^{(2)}}q_{1}(x;\Xi^{x^{(2)}})=0$ leads to 
    \begin{equation*}
    e^{ax^{(2)}}-\beta\big[K_{1}e^{ax^{(2)}}+K_{2}(x^{(2)})e^{\lambda_{2}x^{(2)}}\big]+\mu c=-\lim\limits_{x\uparrow x^{(2)}}J_{xx}(x;\Xi^{x^{(2)}})>0.
    \end{equation*}
    Combining the last formula with
    \begin{equation*}
    \int_{0}^{x^{(2)}}\big[K_{1}ae^{ax}+K_{2}(\bar{x})\lambda_{2}e^{\lambda_{2}x}\big]dx>\int_{0}^{x^{(1)}}\big[K_{1}ae^{ax}+K_{2}(\bar{x})\lambda_{2}e^{\lambda_{2}x}\big]dx
    \end{equation*}
    yields $q_{1}(x^{(2)};\Xi)>0$. 
    Moreover, we have
    \begin{equation*}
    \begin{aligned}
    &e^{a\hat{x}}-\beta\big[K_{1}e^{a\hat{x}}+K_{2}(\hat{x})e^{\lambda_{2}\hat{x}}\big]+\mu c=0,\\
    &\frac{\partial}{\partial x}\big[K_{1}e^{ax}+K_{2}(x)e^{\lambda_{2}x}-cx\big]\le 0,\quad  \forall x\le\hat{x},
    \end{aligned}
    \end{equation*}
    which leads to $q_{1}(\hat{x};\Xi)<0$. Due to the monotonicity of $q_{1}$ in $[\bar{x},\infty)$, there exists $\bar{x}'\in(\hat{x},x^{(2)})$ such that $q_{1}\le 0$ for any $x\le\bar{x}'$ and $\tilde{W}=\{x|x<\bar{x}'\}$.\\
(3) Case $\bar{x}<\hat{x}$ with $K_{2}(\bar{x})\ge 0$. In this case, $q_{0}(x;\Xi)\ge 0$ holds for any $x\in\mathbb{R}$. And we have $J_{xx}(x,\Xi)>0$ for any $x\in\mathbb{R}$, which yields $0=\lim\limits_{x\uparrow \bar{x}}q_{1}(x;\Xi)>\lim\limits_{x\downarrow \bar{x}}q_{1}(x;\Xi)$. Due to the monotonicity of $q_{1}$ in $[\bar{x},\infty)$, there exists $\bar{x}'>\bar{x}$ such that $q_{1}\le 0$ for any $x\le\bar{x}'$. Note that $K_{2}(x)$ has only one root and is negative when $x$ is larger, then we have $\bar{x}<\hat{x}$ and a similar argument to (2) leads to $q_{1}(\hat{x};\Xi)<0$. Thus there exists $\bar{x}'\in(\hat{x},\infty)$ such that 
$q_{1}\le 0$ for any $x\le\bar{x}'$ and $\tilde{W}=\{x|x<\bar{x}'\}$.

Combining the analysis in three cases (1)$\sim$(3), we obtain that case (3) will turn into case (1), while cases (1) and (2) turn into each other. Moreover, for $\bar{x}$ belonging to case (1), the boundary $\bar{x}''$ after two iterations is in $(\hat{x},\bar{x})$. Then we can conclude that the iterated boundaries $\{x_{n}\}$ converges to $\hat{x}$ for any initial $x_{0}$.
\end{proof}
Thanks to the explicit expression of $V(x)$, the pair of q-functions are obtained as 
\begin{equation*}
\begin{aligned}
&q_{0}(x,t,y)=
\begin{cases}
c-\frac{a}{\beta-\mu a-\frac{1}{2}\sigma^{2}a^{2}}e^{ax}-C_{2}\lambda_{2}e^{\lambda_{2}x}, & x<\hat{x},\\
0, &  x\ge \hat{x},
\end{cases}\\
&q_{1}(x,t,y)=
\begin{cases}
0, &   x<\hat{x},\\
e^{ax}-\beta c(x-\hat{x})-\beta V(\hat{x})+\mu c, & x\ge\hat{x}. 
\end{cases}
\end{aligned}
\end{equation*}

Now, we combine {\bf Scheme-3} and the martingale loss approach in (a) to learn the optimal region. Taking advantage of the explicit expressions, we can parameterize the optimal value function and q-functions by the same profile of parameters $\theta=\{\theta_{i}\}_{i=1,2,3}$ given by
\begin{equation*}
\theta_{1}=a,\quad \theta_{2}=\lambda_{2},\quad\theta_{3}=\hat{x}.
\end{equation*}
The optimal value function and q functions are parameterized by
\begin{align*}
&V(x;\theta)=\begin{cases}
c\frac{\theta^{2}_{2}e^{\theta_{1}(x-\theta_{3})}-\theta_{1}^{2}e^{\theta_{2}(x-\theta_{3})}}{\theta_{1}\theta_{2}(\theta_{2}-\theta_{1})}, & x<\theta_{3},\\
c(x-\theta_{3})+c\frac{\theta_{1}+\theta_{2}}{\theta_{1}\theta_{2}}, & x\ge \theta_{3},
\end{cases}\\
&q_{0}(x;\theta)=\begin{cases}
c-c\frac{\theta_{2}e^{\theta_{1}(x-\theta_{3})}-\theta_{1}e^{\theta_{2}(x-\theta_{3})}}{(\theta_{2}-\theta_{1})}, & x<\theta_{3},\\
0, & x\ge \theta_{3},
\end{cases}\\
&q_{1}(x;\theta)=\begin{cases}
0, & x<\theta_{3},\\
e^{\theta_{1}x}-\beta c(x-\theta_{3})-e^{\theta_{1}\theta_{3}}, & x\ge \theta_{3}.
\end{cases}
\end{align*}
For the martingale (\ref{m3_a}), the martingale loss is approximated by
 \begin{equation*}
 \begin{aligned}
 &\frac{1}{2}\mathbb{E}\int_{t}^{T}[M_{T-}-M_{r-}]^{2}dr\\
 \approx&\frac{1}{2}\mathbb{E}\sum\limits_{n=0}^{N-1}\Big[e^{-\beta T}V(X_{T-};\theta)-e^{-\beta t_{n}}V(X_{t_{n}-};\theta)+\sum\limits_{i=n}^{N-1}\big[H_{t_{i}}-e^{-\beta t_{i}}q_{1}(X_{t_{i}};\theta)\Delta t\big]\\&+\sum\limits_{i=n}^{N-1}\big[c_{t_{i}}-e^{-\beta t_{i}}q_{0}(X_{t_{i}-};\theta)(\xi_{t_{i+1}-}-\xi_{t_{i}-})\big]\Big]^{2}\Delta t\\
  \approx&\frac{1}{2}\mathbb{E}\!\!\sum\limits_{n=0}^{N-1}\!\Big[\!-\!e^{\!-\beta t_{n}}V(X_{t_{n}-\!};\theta)\!+\!\!\sum\limits_{i=n}^{N-1}\!\big[H_{t_{i}}\!-\!e^{\!-\beta t_{i}}q_{1}(X_{t_{i}};\theta)\Delta t\big]\!+\!\!\sum\limits_{i=n}^{N-1}\!\big[c_{t_{i}}\!-\!e^{-\beta t_{i}}q_{0}(X_{t_{i}-\!};\theta)(\xi_{t_{i+1}-\!}\!\!-\!\xi_{t_{i}-})\big]\Big]^{2}\!\Delta t.\!
 \end{aligned}
 \end{equation*}
Using stochastic gradient descent leads to 
\begin{equation}
\label{update_theta}
\begin{aligned}
\theta\leftarrow&\theta+\alpha^{\theta}l(m)\Delta t\sum\limits_{n=0}^{N-1}\Big[-\!e^{-\beta t_{n}}V(X_{t_{n}-};\theta)\!+\!\sum\limits_{i=n}^{N-1}\big[H_{t_{i}}\!-\!e^{-\beta t_{i}}q_{1}(X_{t_{i}};\theta)\Delta t\big]\\&\!+\!\sum\limits_{i=n}^{N-1}\big[c_{t_{i}}\!-\!e^{-\beta t_{i}}q_{0}(X_{t_{i}-};\theta)(\xi_{t_{i+1}-}-\xi_{t_{i}-})\big]\Big]*\Big[e^{-\beta t_{n}}\frac{\partial}{\partial \theta}V(X_{t_{n}-};\theta)+\sum\limits_{i=n}^{N-1}e^{-\beta t_{i}}\big[\frac{\partial}{\partial \theta}q_{1}(X_{t_{i}};\theta)\Delta t\\&+\frac{\partial}{\partial \theta}q_{0}(X_{t_{i}-};\theta)(\xi_{t_{i+1}-}-\xi_{t_{i}-})\big]\Big]
\end{aligned}
\end{equation}

Based on Theorem \ref{qforopjq}, we present the offline and off-policy optimal q-learning Algorithm \ref{simplified-ml-example} as below.
\begin{algorithm}[h]
\caption{Offline (off-policy) optimal q-learning algorithm using martingale loss}
\label{simplified-ml-example}
\renewcommand{\algorithmicrequire}{\textbf{INPUT:}}
\begin{algorithmic}
\REQUIRE
Truncation time $T>0$, number of episodes $M$, number of mesh grids $N$, proper parameterization $\phi(x;\theta)$ for $\phi=V,q_{0},q_{1}$ in the optimal case, initial learning rates $\alpha^{\theta}$ and a learning rate schedule function $l(\cdot)$. 
\STATE Initialize $\theta$. Calculate the time step $\Delta t=\frac{T}{N}$ and set $t_{n}: =n\frac{T}{N}, n=0, 1, 2, \cdots, N$.
\FOR {episode $m=1,2, \cdots, M$}
\STATE Simulate a training sample path. Specifically, for each $n<N$, obtain the state-time-control triples $(X_{t_{n}}, t_{n}, \xi_{t_{n}})$ and $(X_{t_{n+1}-}, t_{n+1}, \xi_{t_{n+1}-})$, the reward/penalty $H_{t_{n}}$, and the  control costs $c_{t_{n}}$.

\STATE Update $\theta$ using (\ref{update_theta}).
\ENDFOR
\end{algorithmic}
\end{algorithm}

In the numerical implementation, the training sample is simulated by $x_{0}$ uniformly sampled in $[\theta_{3}-10,\theta_{3}+10]$, with $\theta_{3}$ denoting the current estimate of boundary during training, and the jump in each step is uniformly sampled from $[0.05,0.1]$. In order to stabilize the gradient estimation from samples, we simulate $10$ paths in each episode and use the sample mean of the gradients in (\ref{update_theta}) truncated in $[-10000,10000]$ to update $\theta$. With true values $(\mu,\sigma,a)=(0.25,1,0.1)$ and known $c=1,\beta=0.1$, the truncation time and the time step are set to $T=100$ and $\Delta t=0.02$, the initial learning rates and learning rate schedule function are chosen as $\alpha^{\theta}=(0.00001, 0.0005, 0.01)$ and $l(m)=(1 + 0.001m + 0.009(m - 1000)1_{\{m \ge 1000\}})^{-1}$. The initial value of $\theta$ is $(0.15, 0.4, 0.5)$. We iterate $\theta$ in $2000$ episodes. The learned values of $\theta_{i},i=1,2,3$ in $2000$ episodes are shown in Figure \ref{learningtheta}. The numerical plots illustrate the satisfactory convergence performance of iterated parameters and the errors between the true value function and the learned value function, confirming the effectiveness and efficiency of our proposed RL-algorithm.

\begin{figure}[H]
\centering
\includegraphics[width=0.95\linewidth, height=1\textheight, keepaspectratio]{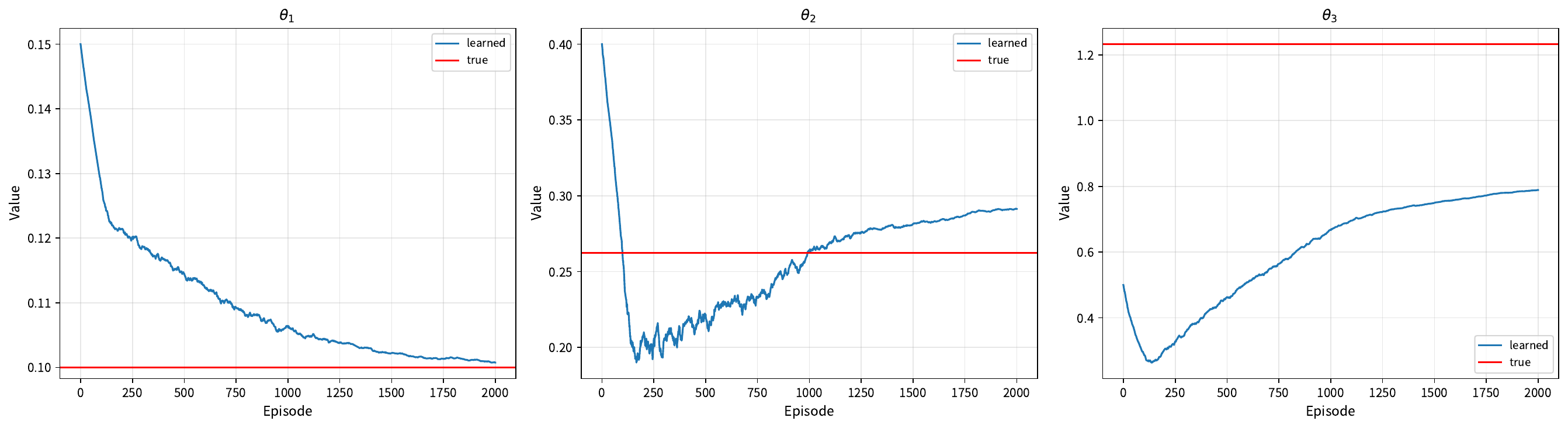}
\caption{Convergence of parameter iterations}
\label{learningtheta}
\end{figure}

\begin{figure}[H]
\centering
\includegraphics[width=0.95\linewidth, height=1\textheight, keepaspectratio]{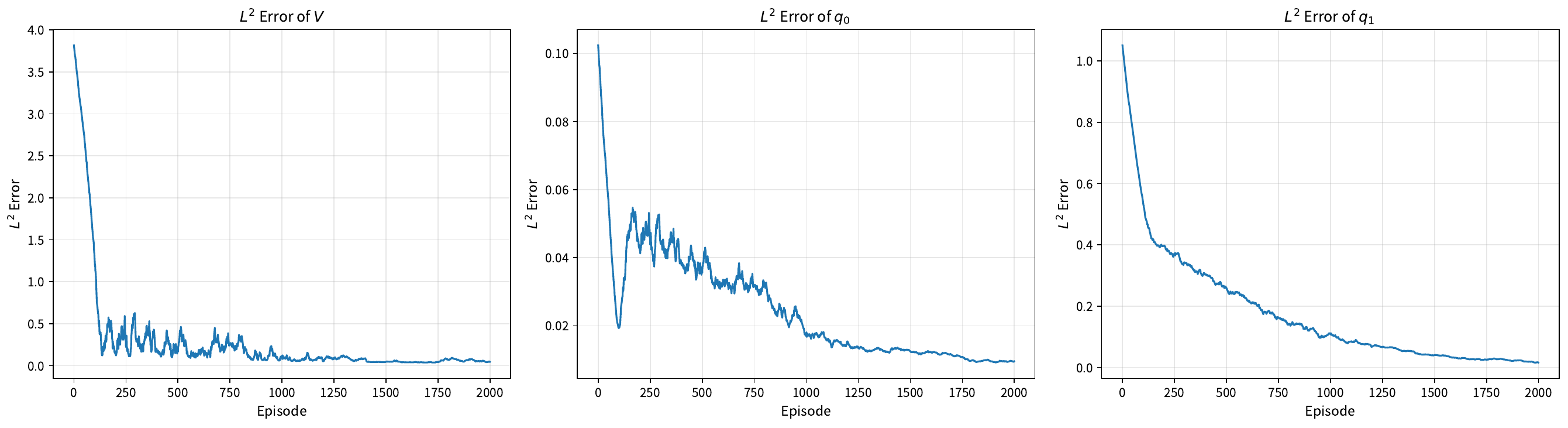}
\caption{$L^{2}$ errors during training}
\label{learningl2}
\end{figure}

\begin{figure}[H]
\centering
\includegraphics[width=0.95\linewidth, height=1\textheight, keepaspectratio]{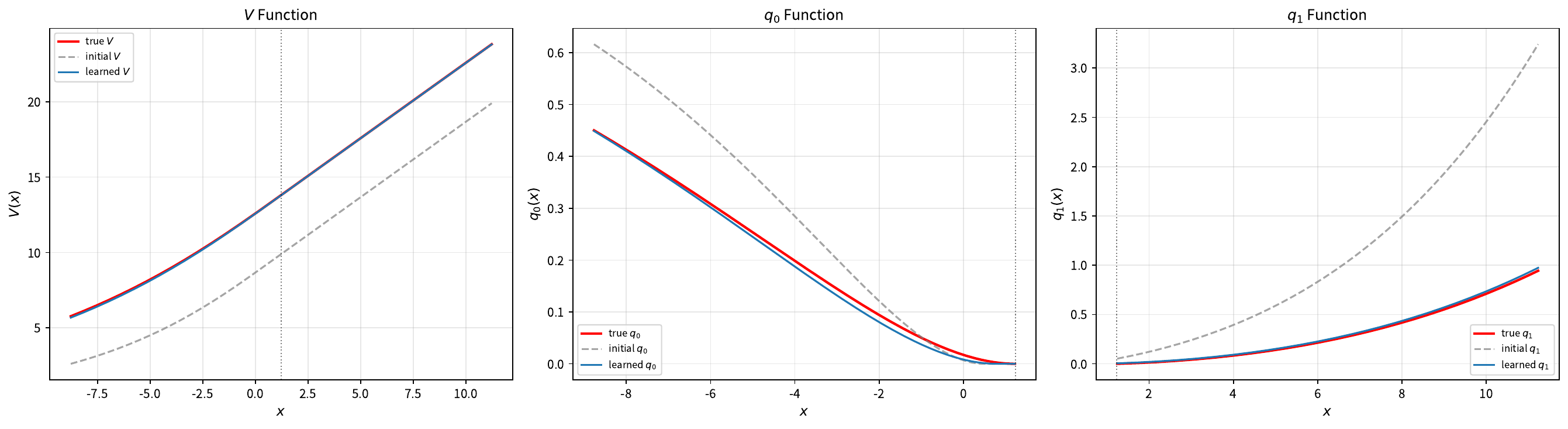}
\caption{Comparison between initial/learned/true functions}
\label{learningfunction}
\end{figure}

\small
\bibliographystyle{apalike}
\bibliography{references}

@article{benevs1980some,
  title={{Some solvable stochastic control problems}},
  author={Bene{\v{s}}, V{\'a}clav E. and Shepp, Larry A. and Witsenhausen, Hans S.},
  journal={Stochastics},
  volume={4},
  number={1},
  pages={39--83},
  year={1980}
}

@article{bather1967sequential,
  title = {{Sequential decisions in the control of a space-ship (finite fuel)}},
  author = {Bather, John and Chernoff, Herman},
  journal = {Journal of Applied Probability},
  volume = {4},
  number = {3},
  pages = {584--604},
  year = {1967}
}

@article{jeanblanc1995optimization,
  title   = {{Optimization of the flow of dividends}},
  author  = {Jeanblanc-Picqué, M. and Shiryaev, A. N.},
  journal = {Russian Mathematical Surveys},
  volume  = {50},
  number  = {2},
  pages   = {257--277},
  year    = {1995}
}

@article{asmussen1997controlled,
  title={{Controlled diffusion models for optimal dividend pay-out}},
  author={Asmussen, S{\o}ren and Taksar, Michael},
  journal={Insurance: Mathematics and Economics},
  volume={20},
  number={1},
  pages={1--15},
  year={1997}
}

@article{kogan2001equilibrium,
  title={{An equilibrium model of irreversible investment}},
  author={Kogan, Leonid},
  journal={Journal of Financial Economics},
  volume={62},
  number={2},
  pages={201--245},
  year={2001}
}

@article{yan2022irreversible,
  title={{Irreversible reinsurance: A singular control approach}},
  author={Yan, Tingjin and Park, Kyunghyun and Wong, Hoi Ying},
  journal={Insurance: Mathematics and Economics},
  volume={107},
  pages={326--348},
  year={2022}
}

@article{ferrari2021optimal,
  title={{An optimal extraction problem with price impact}},
  author={Ferrari, Giorgio and Koch, Torben},
  journal={Applied Mathematics and Optimization},
  volume={83},
  number = {3},
  pages={1951--1990},
  year={2021}
}

@article{wang2020reinforcement,
  title = {{Reinforcement learning in continuous time and space: A stochastic control approach}},
  author = {Wang, Haoran and Zariphopoulou, Thaleia and Zhou, Xun Yu},
  year = {2020},
  journal = {Journal of Machine Learning Research},
  volume = {21},
  number = {198},
  pages = {1--34}
}

@article{wang2020continuous,
  title = {{Continuous-time mean-variance portfolio selection: A reinforcement learning framework}},
  author = {Wang, Haoran and Zhou, Xun Yu},
  year = {2020},
  journal = {Mathematical Finance},
  volume = {30},
  number = {4},
  pages = {1273--1308}
}

@article{dong2024randomized,
  title = {{Randomized optimal stopping problem in continuous time and reinforcement learning algorithm}},
  author = {Dong, Yuchao},
  year = {2024},
  journal = {SIAM Journal on Control and Optimization},
  volume = {62},
  number = {3},
  pages = {1590--1614}
}

@article{dianetti2024exploratory,
  title={Exploratory optimal stopping: A singular control formulation},
  author={Dianetti, Jodi and Ferrari, Giorgio and Xu, Renyuan},
  journal={arXiv preprint arXiv:2408.09335},
  year={2024}
}

@article{dai2024learning,
  title={Learning to optimally stop diffusion processes, with financial applications},
  author={Dai, Min and Sun, Yu and Xu, Zuo Quan and Zhou, Xun Yu},
  journal={arXiv preprint arXiv:2408.09242},
  year={2024}
}

@article{jia2022policy,
  title = {{Policy evaluation and temporal-difference learning in continuous time and space: A martingale approach}},
  author = {Jia, Yanwei and Zhou, Xunyu},
  year = {2022},
  journal = {Journal of Machine Learning Research},
  volume = {23},
  number = {154},
  pages = {1--55}
}

@article{jia2022policy1,
  title = {{Policy gradient and actor-critic learning in continuous time and space: Theory and algorithms}},
  author = {Jia, Yanwei and Zhou, Xunyu},
  year = {2022},
  journal = {Journal of Machine Learning Research},
  volume = {23},
  number = {275},
  pages = {1--50}
}

@article{jia2023q,
  title = {{q-learning in continuous time}},
  author = {Jia, Yanwei and Zhou, Xunyu},
  year = {2023},
  journal = {Journal of Machine Learning Research},
  volume = {24},
  number = {161},
  pages = {1--61}
}

@article{liang2024equilibria,
  title = {{Equilibria for time-inconsistent singular control problems}},
  author = {Liang, Zongxia and Luo, Xiaodong and Yuan, Fengyi},
  year = {2024},
  journal = {SIAM Journal on Control and Optimization},
  volume = {62},
  number = {6},
  pages = {3213--3238}
}

@book{karatzas1998brownian,
  title={{Brownian Motion and Stochastic Calculus}},
  author={Karatzas, Ioannis and Shreve, Steven E.},
  year={1998},
  publisher={Springer}
}

@article{WeiYu,
  title={{Continuous-time q-learning for mean-field control problems}},
  author={Wei, Xiaoli and Yu, Xiang},
  journal={Applied Mathematics $\&$ Optimization},
  volume={91},
year={2025},
  pages={10},
}

@article{BoHuangYu,
  title={{On optimal tracking portfolio in incomplete markets: The reinforcement learning approach}},
  author={Bo, Lijun and Huang, Yijie and Yu, Xiang},
  journal={SIAM Journal on Control and Optimization},
  volume={63},
  number={1},
year={2025},
pages={321-348},
}

@article{DaiDJZ,
  title={{Learning Merton’s strategies in an incomplete market: Recursive
entropy regularization and biased Gaussian exploration}},
  author={Dai, Min and Dong, Yuchao and Jia, Yanwei and Zhou, Xun Yu},
  journal={arXiv preprint,  arXiv:2312.11797},
year={2023}, 
}

@article{Jia24,
  title={Continuous-time risk-sensitive reinforcement learning via quadratic variation penalty},
  author={Jia, Yanwei},
  journal={Preprint, available at https://ssrn.com/abstract=4800185},
year={2024}, 
}

@article{sutton1988learning,
  title = {Learning to predict by the methods of temporal differences},
  author = {Sutton, Richard S.},
  year = {1988},
  journal = {Machine Learning},
  volume = {3},
  number = {1},
  pages = {9--44}
}

@book{oksendal1992stochastic,
  title = {{Stochastic Differential Equations}},
  author = {{\O}ksendal, Bernt},
  year = {1992},
  publisher = {Springer Berlin Heidelberg}
}

@article{dupuis1991lipschitz,
  title = {{On Lipschitz continuity of the solution mapping to the Skorokhod problem, with applications}},
  author = {Dupuis, Paul and Ishii, Hitoshi},
  year = {1991},
  journal = {Stochastics and Stochastic Reports},
  volume = {35},
  number = {1},
  pages = {31--62}
}

@article{anderson1976small,
  title = {Small random perturbation of dynamical systems with reflecting boundary},
  author = {Anderson, Robert F. and Orey, Steven},
  year = {1976},
  journal = {Nagoya Mathematical Journal},
  volume = {60},
  pages = {189--216}
}

@article{huang2025convergence,
  title = {Convergence of policy iteration for entropy-regularized stochastic control problems},
  author = {Huang, Yu-Jui and Wang, Zhenhua and Zhou, Zhou},
  year = {2025},
  journal = {SIAM Journal on Control and Optimization},
  volume = {63},
  number = {2},
  pages = {752--777}
}

@article{tran2025policy,
  title = {Policy iteration for exploratory Hamilton--Jacobi--Bellman equations},
  author = {Tran, Hung Vinh and Wang, Zhenhua and Zhang, Yuming Paul},
  year = {2025},
  journal = {Applied Mathematics \& Optimization},
  volume = {91},
  number = {2},
  pages = {50}
}

@article{ma2025convergence,
  title = {Convergence analysis for entropy-regularized control problems: A probabilistic approach},
  author = {Ma, Jin and Wang, Gaozhan and Zhang, Jianfeng},
  year = {2025},
  journal = {arXiv preprint, arXiv:2406.10959}
}

@article{bjork2017timeinconsistent,
  title = {On time-inconsistent stochastic control in continuous time},
  author = {Bj{\"o}rk, Tomas and Khapko, Mariana and Murgoci, Agatha},
  year = 2017,
  journal = {Finance and Stochastics},
  volume = {21},
  number = {2},
  pages = {331--360}
}

@book{bjork2021time,
  title={{Time-Inconsistent Control Theory with Finance Applications}},
  author={Bj{\"o}rk, Tomas and Khapko, Mariana and Murgoci, Agatha},
  year={2021},
  publisher={Springer}
}

@article{HLYZ25,
  title = {Continuous-time reinforcement learning for optimal switching over multiple regimes},
  author = {Huang, Yijie and Li, Mengge and Yu, Xiang and Zhou, Zhou},
  year = {2025},
  journal = {arXiv preprint, arXiv:2512.04697}
}

@article{WYY25,
  title = {Unified continuous-time q-learning for mean-field game and mean-field control problems},
  author = {Wei, Xiaoli and Yu, Xiang and Yuan, Fengyi},
  year = {2024},
  journal = {arXiv preprint, arXiv:2407.04521}
}

@book{wang2021nonlinear,
  title={{Nonlinear Second Order Parabolic Equations}},
  author={Wang, Mingxin},
  year={2021},
  publisher={CRC Press}
}

@article{zhu1992generalized,
  title = {{Generalized solution in singular stochastic control: The nondegenerate problem}},
  author = {Zhu, Hang},
  year = {1992},
  journal = {Applied Mathematics $\&$ Optimization},
  volume = {25},
  number = {3},
  pages = {225--245}
}

@article{lon2011model,
  title = {{A model for optimally advertising and launching a product}},
  author = {Lon, Pui Chan and Zervos, Mihail},
  year = {2011},
  journal = {Mathematics of Operations Research},
  volume = {36},
  number = {2},
  pages = {363--376}
}

@article{becherer2018optimala,
  title = {{Optimal asset liquidation with multiplicative transient price impact}},
  author = {Becherer, Dirk and Bilarev, Todor and Frentrup, Peter},
  year = {2018},
  journal = {Applied Mathematics $\&$ Optimization},
  volume = {78},
  number = {3},
  pages = {643--676}
}

@article{ferrari2025singular,
  title = {On the singular control of a diffusion and its running infimum or supremum},
  author = {Ferrari, Giorgio and Rodosthenous, Neofytos},
  year = {2025},
  journal = {arXiv preprint, arXiv:2501.17577}
}
\end{document}